\documentclass[reqno,11pt,psamsfonts]{amsart}

\usepackage[all]{xy}
\usepackage{amsthm}
\usepackage{amssymb}
\usepackage{amsmath,amscd}
\usepackage{mathrsfs}
\usepackage[dvips]{graphicx, color}
\usepackage{tikz}

\def\cal{\mathcal}

\def\PP{{\mathbb P}}

\def\ZZ{{\mathbb Z}}

\def\11{{1\kern-3.5pt 1}}
\def\mumu{{\mu\kern-4.2pt\mu}}
\def\boxtimes{\setbox0\hbox{$\Box$}\copy0\kern-\wd0\hbox{$\times$}}

\def\End{\operatorname {E nd}}
\def\Ext{\operatorname {Ext}}

\def\Hom{\operatorname {Hom}}
\def\id{\operatorname {id}}

\def\Ker{\operatorname {ker}}

\def\Spec{\operatorname {Spec}}

\def\Aut{\operatorname{Aut}}

\def\dim{\operatorname{dim}}

\def\End{\operatorname{End}}

\def\Ext{\operatorname{Ext}}

\def\gldim{\operatorname{gldim}}
\def\grmod{\operatorname{grmod}}

\def\Hom{\operatorname{Hom}}

\def\Im{\operatorname{Im}}

\def\Ker{\operatorname{Ker}}

\def\max{\operatorname{max}}

\def\mod{\operatorname{mod}}

\def\Proj{\operatorname{Proj}}

\def\sup{\operatorname{sup}}

\def\tails{\operatorname{tails}}

\def\tors{\operatorname{tors}}

\def\G{\mathop{\underline{\underline{\it \Gamma}}}\nolimits}

\def\l{\leftarrow}

\def\d{\downarrow}

\let\oldtext\text
\def\text#1{\oldtext{\normalshape #1}}

\def\a{\alpha}
\def\b{\beta}
\def\c{\gamma}
\def\d{\delta}
\def\e{\epsilon}

\def\l{\lambda}
\def\s{\sigma}
\def\t{\tau}

\def\G{\Gamma}

\def\cA{{\cal A}}

\def\cM{{\cal M}}
\def\cN{{\cal N}}
\def\cO{{\cal O}}
\def\cP{{\cal P}}

\def\cV{{\cal V}}

\def\cX{{\cal X}}
\def\cY{{\cal Y}}

\def\<{\langle}
\def\>{\rangle}
\def\CM{\operatorname {CM}}
\def\uCM{\underline {\operatorname
{CM}}}

\def\sC{\mathscr C}
\def\sD{\mathscr D}

\def\Specn{\operatorname{Spec_{nc}}}
\def\Projn{\operatorname{Proj_{nc}}}

\def\Ind{\operatorname{Ind}}

\def\normalshape{\rm}

\newtheorem{lemma}{Lemma}[section]
\newtheorem{proposition}[lemma]{Proposition}
\newtheorem{theorem}[lemma]{Theorem}
\newtheorem{corollary}[lemma]{Corollary}

{

}

\theoremstyle{definition}

\newtheorem{example}[lemma]{Example}
\newtheorem{definition}[lemma]{\sl Definition}

\theoremstyle{remark}

\newtheorem{remark}[lemma]{Remark}

\begin{document}

\pagenumbering{arabic}

\title[Noncomm Conics in Calabi-Yau Quantum Projective Planes]{Noncommutative Conics in Calabi-Yau Quantum Projective Planes}

\author[H. Hu \& M. Matsuno \& I. Mori]{Haigang Hu, Masaki Matsuno and Izuru Mori}

\address{Graduate School of Science and Technology, Shizuoka University, Ohya 836, Shizuoka 422-8529, JAPAN}

\email{huhaigang\_phy@163.com}
\email{matsuno.masaki.14@shizuoka.ac.jp }

\address{Department of Mathematics, Faculty of Science, Shizuoka University, Ohya 836, Shizuoka 422-8529, JAPAN}

\email{mori.izuru@shizuoka.ac.jp}



\thanks{The second author was supported by Research Fellowships of the Japan Society for the Promotion of Science for Young Scientists
(No. 21J11303).}

\thanks {The third author was supported by 
Grants-in-Aid for Scientific Research (C) 20K03510 
Japan Society for the Promotion of Science 
and 
Grants-in-Aid for Scientific Research (B) 16H03923 
Japan Society for the Promotion of Science.}


\begin{abstract} 
In noncommutative algebraic geometry, 
noncommutative quadric hypersurfaces
are major objects of study.  In this paper, we focus on studying 
noncommutative conics $\Projn A$ embedded into Calabi-Yau quantum projective planes.  
In particular, we give complete classifications of homogeneous coordinate algebras $A$ of noncommutative conics up to isomorphism of graded algebras, and
of noncommutive conics $\Projn A$ up to isomorphism of noncommutative schemes. \\

\end{abstract}

\maketitle

\section{Introduction} 

Throughout this paper, let $k$ be an algebraically closed field of characteristic $0$.  By Sylvester's theorem, it is elementary to classify (commutative) quadric hypersurfaces in $\PP^{d-1}$, namely,  they are isomorphic to 
$$\Proj k[x_1, \dots, x_d]/(x_1^2+\cdots +x_j^2)\subset \PP^{d-1}$$ 
for some $j=1, \dots, d$.   The ultimate goal of our project is to classify noncommutative quadric hypersurfaces in quantum $\PP^{d-1}$'s defined below.   

Let $A$ be a right noetherian connected graded algebra.  We denote by $\grmod A$ the category of finitely generated graded right $A$-modules.  A morphism in $\grmod A$ is a right $A$-module homomorphism preserving degrees.  For $M, N \in \grmod A$, we write $\Ext^i_A(M, N):=\Ext^{i}_{\grmod A}(M,N)$
for the extension groups in $\grmod A$.

\begin{definition}  A noetherian connected graded algebra $S$ generated in degree $1$ is called a {\it $d$-dimensional quantum polynomial algebra}  if 
\begin{enumerate}
\item{} $\operatorname{gldim} S=d$, 
\item{} $\Ext^i_S(k, S(-j))\cong \begin{cases}  k & \textnormal { if } i=j=d, \\
0 & \textnormal { otherwise, } \end{cases}$ and 
\item{} $H_S(t):=\sum_{i=0}^{\infty}(\dim_kS_i)t^i=1/(1-t)^d$.
\end{enumerate}
\end{definition} 

A $d$-dimensional quantum polynomial algebra $S$ is a noncommutative analogue of the commutative polynomial algebra $k[x_1, \dots, x_d]$, so the {\it noncommutative projective scheme} $\operatorname{Proj_{nc}}S$ associated to $S$ in the sense of \cite{AZ} is regarded as a quantum $\PP^{d-1}$,  
so it is reasonable to define a noncommutative quadric hypersurface in a quantum $\PP^{d-1}$ as follows:

\begin{definition} We say that $A=S/(f)$ is the {\it homogeneous coordinate algebra} of a noncommutative quadric hypersurface if $S$ is a $d$-dimensional quantum polynomial algebra and $f\in S_2$ is a regular normal element.    
\end{definition} 

We often say that $A=S/(f)$ is a {\it noncommutative quadric hypersurface} by abuse of terminology.  It is now reasonable to say that $A=S/(f)$ is a {\it noncommutative conic} if $S$ is a 3-dimensional quantum polynomial algebra.  For every 3-dimensional quantum polynomial algebra $S$, there exists a 3-dimensional Calabi-Yau quantum polynomial algebra $S'$ such that $\grmod S\cong \grmod S'$ by \cite[Theorem 4.4]{IM} so that $\Projn S\cong \Projn S'$, so, in this paper, we will restrict the notion of noncommutative conic as follows:   

\begin{definition}
We say that $A= S/(f)$ is (the homogeneous coordinate algebra of) a 
{\it noncommutative conic (in a Calabi-Yau quantum $\PP^2$)} if $S$ is a $3$-dimensional Calabi-Yau quantum polynomial algebra, and $0\neq f \in Z(S)_2$ is a central element. 
\end{definition} 

Using the above definition, we can describe all the possible homogeneous coordinate algebras of noncommutative conics as follows: 

\begin{theorem} [Corollary \ref{cor.cgc}] Every  homogeneous coordinate algebra $A$ of a noncommutative conic is isomorphic to either 
\begin{enumerate}
\item{} (commutative case) 
$$k[x, y, z]/(x^2), k[x, y, z]/(x^2+y^2), k[x, y, z]/(x^2+y^2+z^2),$$ 
or 
\item{} (noncommutative case)
$$S^{(\alpha, \beta, \gamma)}/(ax^2+by^2+cz^2)$$ 
for some $(a, b, c)\in \PP^2$ where 
$$S^{(\alpha, \beta, \gamma)}:=k\langle x, y, z\rangle/(yz+zy+\alpha x^2, zx+xz+\beta y^2, xy+yx+\gamma z^2)$$
for some $\alpha, \beta, \gamma\in k$ such that $\alpha\beta\gamma=0$ or $\alpha=\beta=\gamma$. 
\end{enumerate}  
\end{theorem} 

The above classification of $A$ makes it possible to classify noncommutative conics $\Projn A$ up to isomorphism of noncommutative schemes by the methods explained below.


If 
$A=S/(f)$ is a noncommutative quadric hypersurface, then there exists a unique central element $f^!\in Z(A^!)_2$ such that $S^!=A^!/(f^!)$ where $A^!$ is the quadratic dual of $A$.  The finite dimensional algebra $C(A):=A^![(f^!)^{-1}]_0$ plays an important role to study $A$ by the following theorem:    

\begin{theorem} [{\cite[Proposition 5.2 (1)]{SV}}]
\label{thm.SV}
If $A$ is a noncommutative quadric hypersurface, then 
the stable category $\uCM^{\ZZ}(A)$ of the category of maximal Cohen-Macaulay graded right $A$-modules and the bounded derived category $\sD^b(\operatorname{mod} C(A))$ of the category of finitely generated right $C(A)$-modules are equivalent. 
\end{theorem}

Let $A$ and $A'$ be noncommutative quadric hypersurfaces. Despite the notation, it is not trivial to see if $A\cong A'$ implies $C(A)\cong C(A')$ from the definition.  By the above theorem, if $\grmod A\cong \grmod A'$,  
then $C(A)$ and $C(A')$ are derived equivalent. 
In this paper, we show that $\grmod A\cong \grmod A'$ implies
$C(A)\cong C(A')$ (Lemma \ref{lem.cpn}), 
so it is reasonable to classify $C(A)$, which we will do in the case $d=3$.   

On the other hand, the point variety is an important invariant in noncommutative algebraic geometry.  In this paper, we show that every noncommutative conic $A$ has a point variety $E_A\subset \PP^2$ in the sense of this paper (Proposition \ref{prop.g1}).  The point variety $E_A$ plays an essential role in completing the classification of $C(A)$.   
The following is a summary of the classifications of $C(A)$ and $E_A$ (Theorem \ref{thm.main2}): 

\begin{table}[h] 
\begin{center}
\begin{tabular}{|c|c|c|}  \hline
$C(A)$    &  $E_A$ & An example of $A$ \\ 
\hline \hline
$k\langle u,v\rangle /(uv+vu, u^2,v^2)$ & {\rm a line}  &  $k[x,y,z]/(x^2)$ \\ \hline
$k\langle u,v\rangle /(uv+vu, u^2,v^2-1)$ & {\rm two lines} &  $k[x,y,z]/(x^2 + y^2)$\\ \hline
$\operatorname{M}_2(k)$ & {\rm  a smooth conic} &  $k[x,y,z]/(x^2 + y^2 + z^2)$\\ \hline
$k[u,v]/(u^2,v^2)$  & {\rm  a line}  &  $S^{(0, 0, 0)}/(x^2)$ \\ \hline 
$k[u]/(u^4)$  & $1$ {\rm point}  &  $S^{(1, 1, 0)}/(x^2)$ \\ \hline 
$k[u]/(u^3) \times k$ & $2$ {\rm points}  &  $S^{(1, 1, 0)}/(3x^2 + 3y^2 + 4z^2)$  \\ \hline 
$k[u]/(u^2) \times k[u]/(u^2)$ & $3$ {\rm points} &  $S^{(0, 0, 0)}/(x^2+y^2)
$\\ \hline 
$k[u]/(u^2) \times k^2$ & $4$ {\rm points} & $S^{(1, 1, 0)}/(x^2 + y^2 - 4 z^2)$ \\ \hline 
$k^4$ & $6$ {\rm points} & $S^{(0, 0, 0)}/(x^2+y^2+z^2)
$\\ \hline 
\end{tabular}
\end{center}
\end{table}

Let $A$ and $A'$ be noncommutative conics.  We show that $\uCM^{\ZZ}(A)\cong \uCM^{\ZZ}(A')$ if and only if $C(A)\cong C(A')$ (Theorem \ref{thm.mcc}), and that $\Projn A\cong \Projn A'$ implies $C(A)\cong C(A')$ (Theorem \ref{thm.GA}).   Together with the above classifications of $A$ and $C(A)$, we can conclude that there are exactly 9 isomorphism classes of homogeneous coordinate algebras $A$ of noncommutative conics, and that there are exactly 9 isomorphism classes of noncommutative conics $\Projn A$ (Theorem \ref{thm.HMM}).
 
\section{Preliminaries} 

Throughout this paper, we fix an algebraically closed field $k$ of characteristic 0.  
All algebras and schemes in this paper are defined over $k$.  

In this paper, a variety means a reduced scheme (which can be reducible).  For a scheme $X$, we denote by $\overline X$ when we view $X$ as a variety (a scheme $X$ with the reduced induced structure).    
We denote by $\#(X)$ the number of points of $\overline X$ (without counting multiplicity). 
For 
$\s\in \Aut X$, we denote by $\bar \s\in \Aut \overline X$ the automorphism of a variety $\overline X$ induced by $\s\in \Aut X$.
For a variety $X\subset \PP^{n-1}$ and $\s\in \Aut X$, we denote by 
$\Delta _{X, \s}:=\{(p, \s(p))\in \PP^{n-1}\times \PP^{n-1}\mid p\in X\}$ the graph of $X$ by $\s$. 

For $f=\sum_{1\leq i, j\leq n}a_{ij}x_ix_j\in k\<x_1, \dots, x_n\>_2$, we define $\tilde f:=\sum_{1\leq i, j\leq n}a_{ij}x_{ij}\in k[x_{11}, x_{12}, \dots, x_{nn}]_1$.  We denote by $\Sigma :\PP^{n-1}\times \PP^{n-1}\to \PP^{n^2-1}$ the Segre embedding.  
For a finite dimensional vector space $V$, we write $U:=V\otimes V$, so that if $(p, q)\in \PP(V^*)\times \PP(V^*)$, then $\Sigma (p, q)\in \PP(U^*)$. (If $V=\sum_{i=1}^nkx_i$, then $U=\sum_{i, j=1}^nkx_{ij}$ where $x_{ij}:=x_i\otimes x_j$.)   Moreover, for a subspace $W\subset V^{\otimes 2i}$, we write $\widetilde W$ when we view $W$ as a subspace of $U^{\otimes i}$.  By this notation, $\cV(W)\subset \PP(V^*)^{\times 2i}$ while $\cV(\widetilde W)\subset \PP(U^*)^{\times i}$.  We often use the following lemma.

\begin{lemma} \label{lem.tilw} 
Let $V$ be a finite dimensional vector space and $U=V\otimes V$.  For a subspace $W\subset V^{\otimes 2i}$, $\cV(\widetilde W)\cap \Im(\Sigma^{\times i})=\Sigma^{\times i}(\cV(W))$.  
\end{lemma} 

\begin{definition} Let $S=k\<x_1, \dots, x_n\>/(f_1, \dots, f_m)$ be a quadratic algebra.  We define 
$$\G=\G_S:=\Sigma^{-1}(\Proj k[x_{11}, x_{12}, \dots, x_{nn}]/(\tilde f_1, \dots, \tilde f_m))\subset \PP^{n-1}\times \PP^{n-1}.$$
\begin{enumerate}
\item{} We say that $S$ satisfies (G1)$^+$ if there exists a scheme $E\subset \PP^{n-1}$ such that both projections $p_1, p_2:\PP^{n-1}\times \PP^{n-1}\to \PP^{n-1}$ restrict to isomorphisms $p_1|_{\G}, p_2|_{\G}:\G\to E$.  In this case, we call $E$ the point scheme of $S$,
and write $\cP^+(S)=(E, \s)$ 
where $\s:=p_2|_{\G}(p_1|_{\G})^{-1}\in \Aut E$. 
\item{} We say that $S$ satisfies (G1) if there exists a variety $E\subset \PP^{n-1}$ and an automorphism of a variety $\s\in \Aut E$ such that 
$\overline {\G}=\Delta_{E, \s}$.
In this case, we call $E$ the point variety of $S$, and write $\cP(S)=(E, \s)$. 
\item{} We say that $S$ satisfies (G2) if  
$$S=k\<x_1, \dots, x_n\>/(f\in k\<x_1, \dots, x_n\>_2\mid f|_{\overline \G}=0).$$    
\end{enumerate}
\end{definition} 

\begin{remark} \label{rem.g1g2} 
Let $S=k\<x_1, \dots, x_n\>/(f_1, \dots, f_m)$ be a quadratic algebra. Since $\overline \G=\{(p, q)\in \PP^{n-1}\times \PP^{n-1}\mid f_1(p, q)=\cdots =f_m(p, q)=0\}$, if $S$ satisfies (G1)$^+$ with $\cP^+(S)=(E, \s)$, then $S$ satisfies (G1) with $\cP(S)=(\overline {E}, \overline {\s})$.  It follows that $\cP^+(S)=\cP(S)$ if and only if $E$ is a reduced scheme (a variety). 
\end{remark} 

\begin{lemma} [{\cite[Lemma 2.5(1)]{MU1}}]\label{lem.es} 
Let $S, S'$ be quadratic algebras satisfying (G1) with $\cP(S)=(E, \s), \cP(S')=(E', \s')$.  If $S\cong S'$ as graded algebras, then $E\cong E'$ as varieties and $|\s|=|\s'|$. 
\end{lemma} 

\begin{lemma} \label{lem.g1c}  
If $A$ is a commutative quadratic algebra, then $A$ satisfies (G1)$^+$ with $\cP^+(A)=(\Proj A, \id)$. 
In particular, $A$ satisfies (G1) with $\cP(A)=(\overline{\Proj A}, \id)$.  
\end{lemma} 

\begin{proof} First, let $S=k[x_1, \dots, x_n]$ so that $\Proj S=\PP^{n-1}$ and $\Im \Sigma=\Proj (S\circ S)$ where $S\circ S$ is the Segre product.  Since $(S\circ S)/(x_{ij}-x_{ji})_{1\leq i, j\leq n}\cong S^{(2)}$, we have a natural isomorphism  
\begin{align*}
\phi: \G_S & \to \Proj k[x_{11}, \dots, x_{nn}]/(x_{ij}-x_{ji})_{1\leq i, j\leq n}\cap \Im \Sigma \\ 
& \to \Proj (S\circ S)/(x_{ij}-x_{ji})_{1\leq i, j\leq n} \\
& \to \Proj S^{(2)}\to \Proj S=\PP^{n-1}
\end{align*}  
by Lemma \ref{lem.tilw}.  If $p_1, p_2:\PP^{n-1}\times \PP^{n-1}\to \PP^{n-1}$ are projections, then we can see that $p_1|_{\G_S}=p_2|_{\G_S}=\phi$, so $S$ satisfies (G1)$^+$ with $\cP^+(S)=(\PP^{n-1}, \id)$. 

If $A=S/(f_1, \dots, f_m)$ is a commutative quadratic algebra, then we can see that the isomorphism $(S\circ S)/(x_{ij}-x_{ji})_{1\leq i, j\leq n}\to S^{(2)}; \; \tilde f\mapsto \bar f$ induces an isomorphism 
$$(S\circ S)/(x_{ij}-x_{ji}, \tilde f_1, \dots, \tilde f_m)\to  S^{(2)}/(\bar f_1, \dots, \bar f_m) 
\cong A^{(2)},$$ 
so $\phi$ restricts to an isomorphism  
\begin{align*}
\phi|_{\G_A}: \G_A & \to \Proj k[x_{11}, \dots, x_{nn}]/(x_{ij}-x_{ji}, \tilde f_1, \dots, \tilde f_m)\cap \Im \Sigma \\ 
& \to  \Proj (S\circ S)/(x_{ij}-x_{ji}, \tilde f_1, \dots, \tilde f_m) \\
& \to \Proj A^{(2)} \to  \Proj A\subset \PP^{n-1}
\end{align*}  
by Lemma \ref{lem.tilw}.  Since $p_1|_{\G_A}=p_2|_{\G_A}=\phi|_{\G_A}$, $A$ satisfies (G1)$^+$ with $\cP^+(A)=(\Proj A, \id)$. 
\end{proof} 

\begin{theorem} [{\cite{ATV}}]
Every 3-dimensional quantum polynomial algebra $S$ satisfies (G1)$^+$. 
The point scheme of a 3-dimensional quantum polynomial algebra is either $\PP^2$ or a cubic divisor in $\PP^2$.  
\end{theorem} 

If $S$ is a 3-dimensional quantum polynomial algebra with $\cP^+(S)=(E, \s)$, then $S$ can be recovered by the pair $(E, \s)$ by \cite{ATV}, so we write $S=\cA^+(E, \s)$.  
We define a type of a $3$-dimensional quantum polynomial algebra $S=\cA^+(E, \s)$ in terms of the point scheme $E\subset \PP^2$.  
\begin{description}
\item[{\rm Type P}] 
      $E$ is $\mathbb{P}^{2}$. 
\item [{\rm Type S}] 
      $E$ is a triangle. 
\item[{\rm Type S'}] 
      $E$ is a union of a line and a conic meeting at two points. 
\item[{\rm Type T}]
      $E$ is a union of three lines meeting at one point. 
\item[{\rm Type T'}]
      $E$ is a union of a line and a conic meeting at one point. 
\item[{\rm Type NC}]
          $E$ is a nodal cubic curve. 
\item[{\rm Type CC}] $E$ is a cuspidal cubic curve. 
\item[{\rm Type TL}] $E$ is a triple line. 
\item[{\rm Type WL}] $E$ is a union of a double line and a line. 
\item[{\rm Type EC}] $E$ is an elliptic curve. 
\end{description}

\begin{remark} \label{rem.g2g1} 
Let $S=\cA^+(E, \s)$ be a $3$-dimensional quantum polynomial algebra.  By Remark \ref{rem.g1g2}, 
$\cP(S)=(E, \s)$ if and only if $E$ is reduced, that is, $S$ is not of Type TL nor Type WL. 
In this case, $S$ satisfies (G2). 
\end{remark} 

\begin{lemma} 
\label{lem.SP}
Let $S=\mathcal{A}^+(E,\sigma)=k\langle x,y,z \rangle/(f_{1}, f_{2}, f_{3})$ be a $3$-dimensional Calabi-Yau quantum polynomial algebra.   
Then Table $1$ below gives a list of defining relations 
$f_{1}, f_{2}, f_{3}$ up to isomorphism of the algebra $S$ and the corresponding 
$\overline E$ and $\bar \sigma^i$.  In Table 1 below, $\theta=\begin{pmatrix} 0 & 1 & 0 \\ 0 & 0 & 1 \\ 1 & 0 & 0 \end{pmatrix}$ and $\e\in k$ is a primitive 3rd root of unity.  Note that Type S is divided into Type S$_1$ and Type S$_3$ in terms of the form of $\s$.  
\end{lemma}

\begin{center}
{\setlength{\tabcolsep}{1pt}
\renewcommand\arraystretch{1.1}
{\small
\begin{table}
\begin{tabular}{|c|p{5cm}|p{2.7cm}|p{5.3cm}|}
\multicolumn{4}{c}{Table $1$}
\\ [2pt]
\hline  
       Type
           & 
           $f_{1}, f_{2}, f_{3}$
           & \rule{0pt}{8pt} $\overline E$
           & $\bar \sigma^i$
            \\   \hline\hline 
$\rm{P}$ 
            & 
           $
           \left\{
             \begin{array}{ll}
             yz-\alpha zy & \\
             zx-\alpha xz & \ \alpha^{3}=1\\
             xy-\alpha yx & 
             \end{array}
             \right.
             $
           & $ \mathbb{P}^{2}$
           &  
           $\sigma^i(a,b,c)=(a,\alpha ^ib,\alpha^{-i}c)$
           \\ \hline
$\rm{S}_1$ 
           & 
           $\left\{
             \begin{array}{ll}
              yz-\alpha zy &  \\
              zx-\alpha xz & \  \alpha^{3}\neq 0,1\\
              xy-\alpha yx & 
             \end{array}
             \right.
             $
           & $\begin{array}{ll}
           & \cV(x) \\
           \cup & \cV(y) \\
           \cup & \cV(z)
           \end{array}$
           & 
           $\left\{
             \begin{array}{ll}
              \sigma^i(0,b,c)=(0,b,\alpha^i c)\\
              \sigma^i(a,0,c)=(\alpha^i a,0,c)\\
               \sigma^i(a,b,0)=(a,\alpha^i b,0)
             \end{array}
             \right.
             $
           \\ \hline
$\rm{S}_3$ 
          & 
           $\left\{
             \begin{array}{ll}
             zy-\alpha x^{2} &  \\
             xz-\alpha y^{2} & \ \alpha^{3}\neq 0,1\\
             yx-\alpha z^{2} & 
             \end{array}
             \right.
             $
           & $\begin{array}{ll}
           & \cV(x) \\
           \cup & \cV(y) \\
           \cup & \cV(z)
           \end{array}$
           & 
           $\left\{
             \begin{array}{ll}
              \sigma^{i}(0,b,c)=(0, b, \alpha^i c)\theta^i \\
              \sigma^{i}(a,0,c)=(\alpha^i a,0, c)\theta^i \\
              \sigma^{i}(a,b,0)=(a,\alpha^i b, 0)\theta^i 
             \end{array}
             \right.
             $
           \\ \hline
$\rm{S'}$ 
            &
            $\left\{
             \begin{array}{ll}
               yz-\alpha zy+x^{2} & \\
               zx-\alpha xz & \ \alpha^{3}\neq 0,1\\
               xy-\alpha yx &
             \end{array}
             \right.
             $
            & $\begin{array}{ll}
           & \cV(x) \\
           \cup & \cV(x^2-\l yz) \\
           & \l=\frac{\a^3-1}{\a}
           \end{array}$
            & 
            $\left\{
             \begin{array}{ll}
             \sigma^i(0,b,c)=(0,b, \alpha^i c)\\
             \sigma^i(a,b,c)=(a,\alpha^i b,\alpha^{-i} c)\\
             \end{array}
             \right.
             $
            \\[12.1pt] \hline
$\rm{T}$ 
          & 
            $\left\{
             \begin{array}{ll}
             yz-\a zy+x^{2} & \\
             zx-\a xz+y^{2} & \ \alpha^{3}=1\\
             xy-\a yx &
             \end{array}
             \right.
             $
            & $\begin{array}{ll}
           & \cV(x+y) \\
           \cup & \cV(\e x+y) \\
           \cup & \cV(\e^2x+y)
           \end{array}$
             & 
            $\left\{
             \begin{array}{ll}
             \sigma^i(a,-a,c)=\\(a,-\a^ia, i\a^{2i-2}a+\a^{2i}c)\\
             \sigma^i(a,-\e a,c)=\\(a,-\a^i\e a, i\a^{2i-2}\e^2a+\a^{2i}c)\\
             \sigma^i(a,-\e^2 a,c)=\\(a,-\a^i\e^2a, i\a^{2i-2}\e a+\a^{2i}c)
             \end{array}
             \right.
             $
           \\ \hline
$\rm{T'}$ 
          & 
            $\left\{
             \begin{array}{ll}
            yz-zy+xy+yx\\
            zx-xz+x^{2}-yz-zy+y^{2}\\
            xy-yx-y^{2}
             \end{array}
             \right.
             $
            & $\begin{array}{ll}
           & \cV(y) \\
           \cup & \cV(x^2-yz)
           \end{array}$
            & 
            $\left\{
             \begin{array}{ll}
             \sigma^i(a,0,c)=(a,0,ia+c), \\
             \sigma^i(a,b,c)=\\(a-ib,b,-2ia+i^2b+c) 
             \end{array}
             \right.
             $
            \\ \hline
$\rm{NC}$ 
           & 
            $\left\{
             \begin{array}{ll}
               yz-\alpha zy+x^{2} & \\
               zx-\alpha xz+y^{2} & \ \alpha^{3}\neq 0,1\\
               xy-\alpha yx &
             \end{array}
             \right.
             $
          & $\begin{array}{ll}
          \mathcal{V}(x^{3}+y^{3} \\
            \quad -\l xyz) \\
           \l=\frac{\a^3-1}{\a}\end{array}$
            & 
             $
             \begin{array}{ll}
             \sigma^i(a,b,c)=\\
             (a,\alpha^i b,
             -\frac{\alpha^{3i}-1}{\a^{i-1}(\a^3-1)}
             \frac{a^{2}}{b}+\alpha^{2i}c) \\
             \end{array}
             $
           \\[12.1pt] \hline
$\rm{CC}$ 
         & 
         $\left\{
             \begin{array}{ll}
            yz-zy+y^{2}+3x^{2}\\
            zx-xz+yx+ xy-yz-zy\\
            xy-yx-y^{2}
             \end{array}
             \right.
             $
          & $\mathcal{V}(x^{3}-y^{2}z)$
          & $
             \begin{array}{ll}
               \sigma^i(a,b,c)=\\
               (a-ib,b,
               -3i\frac{a^2}{b}+3i^2a-i^3b+c) 
             \end{array}
             $
          \\ \hline
$\rm{TL}$
           & 
             $\left\{
             \begin{array}{ll}
             yz-\alpha zy +x^{2} & \\
             zx-\alpha xz & \ \alpha^{3}=1\\
             xy-\alpha yx
             \end{array}
             \right.
             $
           & $ \mathcal{V}(x^{3})$
            & 
             $\bar \s^i(0, b, c)=(0, b, \a^ic)$
             \\ \hline
$\rm{WL}$ 
             & 
            $\left\{
             \begin{array}{ll}
             yz-zy+y^{2}
             \vspace{0.5em} \\
             zx-xz+yx+xy\\
             xy-yx
             \end{array}
             \right.
             $
& $\begin{array}{ll}
           & \cV(x^2) \\
           \cup & \cV(y)
           \end{array}$
& 
            $\left\{
             \begin{array}{ll}
             \bar \sigma^i(0,b,c)=(0,b,-ib+c)\\
             \bar \sigma^i(a,0,c)=(a,0,c)\\
             \end{array}
             \right.
             $
            \\ \hline 
EC 
             & 
             \begin{tabular}{l}
            $\left\{
             \begin{array}{ll}
              \a yz+\b zy+\c x^{2} \\
              \a zx+\b xz+\c y^{2} \\
              \a xy+\b yx+\c z^{2}
             \end{array}
             \right.$ \\
             $(\a^3+\b^3+\c^3)^3\neq (3\a\b\c)^3$, \\
             $\a\b\c\neq 0$
             \end{tabular}      
            & $\begin{array}{ll}
            \mathcal{V}(x^{3}+y^{3}+z^{3} \\
            \quad -\lambda xyz) \\  
            \lambda=\frac{\a^3+\b^3+\c^3}{\a\b\c}
            \end{array}$
            & 
              $\s_p$ where $p=(\a,\b,\c)\in E$
            \\ \hline
\end{tabular}
\end{table}
}}
\end{center}

\begin{proof} If $S$ is of Type T, TL, WL, this follows by direct calculations (see Example \ref{ex.T} below).   
For the rest,  
$f_1, f_2, f_3$ and $E$ are given in \cite [Lemma 3.3]{IMo}, and $\s^i$ are given in the proof of \cite[Theorem 3.4]{IMo}. 
\end{proof}
 
\begin{example} \label{ex.T}
We give some calculations to show Table 1 for Type T. 
Let 
$S=k\langle x,y,z \rangle/(f_{1}, f_{2}, f_{3})$ be a 3-dimensional Calabi-Yau quantum polynomial algebra of Type T
where 
$$f_1=yz-\a zy+x^2, \;  
            f_2=zx-\a xz+y^2, \; 
            f_3=xy-\a yx,\; \; \; (\a^3=1)$$
            as in Table 1.         
Let $E=\cV(x^3+y^3)=\cV(x+y)\cup \cV(\e x+y)\cup \cV(\e^2x+y)\subset \PP^2$ and define a map $\s:E\to E$ by $$\sigma(a,b,c)=\begin{cases} (a,\a b, -\frac{a^2}{b}+\a^{2}c) & \textnormal { if } (a, b, c)\neq (0, 0, 1) \\
            (0, 0, 1) & \textnormal { if } (a, b, c)=(0, 0, 1). \end{cases}$$  
Since $\a^3=1$ and $a^3+b^3=0$,   
\begin{align*}
f_1(p, \s(p)) &= f_1((a, b, c), (a,\a b,
               -\frac{a^2}{b}+\a^{2}c)) \\
               &= \;  b(-\frac{a^2}{b}+\a^{2}c)-\a c(\a b)+a^2=0 \\
f_2(p, \s(p)) &= f_2((a, b, c), (a,\a b,
               -\frac{a^2}{b}+\a^{2}c)) \\
               &= \;  ca-\a a(-\frac{a^2}{b}+\a^{2}c)+b(\a b)=0 \\
f_3(p, \s(p)) &= f_3((a, b, c), (a,\a b,
               -\frac{a^2}{b}+\a^{2}c)) \\
               &= \;  a(\a b)-\a ba=0 
\end{align*}
so $\cP(S)=(E, \s)$.  We will show that $\sigma^i(a,b,c)=(a,\a^ib, -i\a^{2i-2}\frac{a^2}{b}+\a^{2i}c)$ by induction.  
Since 
\begin{align*}
\sigma^{i+1}(a,b,c) & =\s(a,\a^ib, -i\a^{2i-2}\frac{a^2}{b}+\a^{2i}c) \\
& =\s(a,\a(\a^ib), -\frac{a^2}{\a^ib}+\a^2(-i\a^{2i-2}\frac{a^2}{b}+\a^{2i}c)) \\
& =(a,\a^{i+1}b, -(i+1)\a^{2i}\frac{a^2}{b}+\a^{2i+2}c),
\end{align*}
the result follows.  By plugging $b=-a, -\e a, -\e^2a$, we obtain $\s^i$ in Table 1.  
\end{example}

\section{Classification of $A$} 

 In this section, we classify homogeneous coordinate algebras  $A=S/(f)$ of noncommutative conics up to isomorphism of graded algebras where $S$ are 3-dimensional Calabi-Yau quantum polynomial algebras and $0\neq f\in Z(S)_2$.  It is not easy to check which $S$ has a property that $Z(S)_2\neq 0$ by algebraic calculations especially if $S$ does not have a PBW basis (see \cite{H}), so we use a geometric method in this paper.  

\begin{lemma} \label{lem.deg1} 
Let $S$ be a quadratic algebra and $0\neq u\in Z(S)_1$.
For a variety $E$ and $\s\in \Aut E$, if $\Delta _{E, \s}\subset \overline \G_S$, then $\s|_{E\setminus \cV(u)}=\id$. 
\end{lemma} 

\begin{proof} We may assume that $S=k\<x_1, \dots, x_n\>/I$ and $u=x_1\in Z(S)_1$.  Since $x_1x_i=x_ix_1$ in $S$, $x_1x_i-x_ix_1\in I_2$ for every $i=1, \dots, n$.  For $p=(a_1, \dots, a_n)\in E\setminus \cV(x_1)$ and $\s(p)=(b_1, \dots, b_n)\in E$, we have $a_1b_i=a_ib_1$ for every $i=1, \dots, n$.  Since $a_1\neq 0$, $\s(p)=((b_1/a_1)a_1, (b_1/a_1)a_2, \dots, (b_1/a_1)a_n)$, so $b_1\neq 0$ and $\s(p)=p\in E\setminus \cV(x_1)$.  
\end{proof}  

Let $X, Y$ be irreducible varieties and $\s\in \Aut X$.  
If $X\not \subset Y$, then $X\setminus Y$ is a non-empty open subset in $X$, so $\s|_{X\setminus Y}=\id$ implies $\s=\id$ (\cite[Lemma 4.1, Chapter 1]{Ha}). 

It is now easy to see which $S$ has a property that $Z(S)_1\neq 0$.  

\begin{lemma} Let $S=\cA^+(E, \s)$ be a 3-dimensional Calabi-Yau quantum polynomial algebra.  Then $Z(S)_1\neq 0$ if and only if $\bar \s=\id$ if and only if $S$ is isomorphic to either
\begin{enumerate}
\item{} $S=k[x, y, z]$, and, in this case, $Z(S)_1=S_1$ (Type P), or 
\item{} $S=k\<x, y, z\>/(yz-zy+x^2, zx-xz, xy-yx)$, and, in this case, $Z(S)_1=kx$ (Type TL).
\end{enumerate}
\end{lemma}  

\begin{proof} If $E$ is reduced, then, for every $0\neq u\in Z(S)_1$, $E\not \subset \cV(u)$, so there exists an irreducible component $E_0$ of $E$ such that $E_0\not\subset \cV(u)$.  Since $\s|_{E\setminus \cV(u)}=\id$ by Lemma \ref{lem.deg1}, $\s|_{E_0\setminus \cV(u)}=\id$, so $\s|_{E_0}=\id$.  By Lemma \ref{lem.SP}, we can check that $\s=\id$,
and $S\cong k[x, y, z]$ (Type P).  In this case, $Z(S)_1=S_1$.

If $S$ is of Type TL, then we can show that $Z(S)_1\neq 0$ if and only if $S\cong k\<x, y, z\>/(yz-zy+x^2, zx-xz, xy-yx)$, and, in this case, we can show that $Z(S)_1=kx$ by direct calculations. 

If $S$ is of Type WL, then we can show that $Z(S)_1=0$ by direct calculations.
\end{proof} 

To check which $S$ has a property that $Z(S)_2\neq 0$, we need more work.    

\begin{lemma} \label{lem.2Ve} 
If $S=T(V)/(R)$ is a quadratic algebra, 
$$W=R\otimes V^{\otimes 2}+V\otimes R\otimes V+V^{\otimes 2}\otimes R\subset V^{\otimes 4},$$ 
$U=V\otimes V$, $\overline U=U/\widetilde R$, and $\pi :U\to \overline U$ is a natural surjection, 
then we have the following commutative diagram 
$$\begin{CD}
\overline \G_{S^{(2)}} @>>> \PP(\overline U^*)\times \PP(\overline U^*) \\
@V\cong VV @VV{\pi^*\times \pi^*}V \\
\cV(\widetilde W) @>>> \PP(U^*)\times \PP(U^*).  
\end{CD}$$ 
\end{lemma}

\begin{proof} Since $\Ker (\pi\otimes \pi)=\widetilde R\otimes U+U\otimes \widetilde R\subset \widetilde W$, we see that $\pi\otimes \pi:U\otimes U\to \overline U\otimes \overline U$ induces an isomorphism $(U\otimes U)/\widetilde W\to (\overline U\otimes \overline U)/\overline W$ where $\overline W:=(\pi\otimes \pi)(\widetilde W)$.  
Since 
\begin{align*}
(\pi^*\times \pi^*)(\cV(\overline W)) & =\{(\pi^*\times \pi^*)(p, q))\mid (p, q)\in \cV(\overline W)\} \\
& =\{(\pi^*(p), \pi^*(q))\mid ((\pi\otimes \pi)(\tilde w))(p, q)=0 \; \forall \tilde w\in \widetilde W\} \\
& =\{(\pi^*(p), \pi^*(q))\mid \tilde w((\pi^*(p), \pi^*(q))=0 \; \forall \tilde w\in \widetilde W\} \\
& \subset \cV(\widetilde W),
\end{align*} 
we have a commutative diagram 
$$\begin{CD}
\cV(\overline W) @>>> \PP(\overline U^*)\times \PP(\overline U^*) \\
@V\cong VV @VV{\pi^*\times \pi^*}V \\
\cV(\widetilde W) @>>> \PP(U^*)\times \PP(U^*).  
\end{CD}$$ 
Since $S^{(2)}$ is a quadratic algebra by \cite[Proposition 2.2, Chapter 3]{PolPos} such that $S^{(2)}_1=\overline U, S^{(2)}_2=(\overline U\otimes \overline U)/\overline W$, we see that $S^{(2)}=T(\overline U)/(\overline W)$, so $\overline \G_{S^{(2)}}=\cV(\overline W)$, hence the result. 
\end{proof}

\begin{proposition} \label{prop.ve2} 
If $S=T(V)/(R)$ is a quadratic algebra satisfying (G1) with $\cP(S)=(E, \s)$ and $\pi:U:=V\otimes V\to U/\widetilde R$ is a natural surjection, then $\Delta _{E^{(2)}, \s^{(2)}}\subset \overline \G_{S^{(2)}}$ where  
$\widetilde E:=\Sigma(\overline \G_S)$, $E^{(2)}:=(\pi^*)^{-1}(\widetilde E)$, and $\tilde \s\in \Aut \widetilde E, \s^{(2)}\in \Aut E^{(2)}$ are given by the commutative diagram
$$\begin{CD}
\overline \G_S @>\s^2\times \s^2>\cong > \overline \G_S \\
@V\Sigma V\cong V @V\cong V\Sigma V \\
\widetilde E @>\tilde \s>\cong > \widetilde E \\
@A\pi^*A\cong A @A\cong A\pi^* A \\
E^{(2)} @>\s^{(2)}>\cong > E^{(2)}
\end{CD}$$
\end{proposition} 

\begin{proof} 
Since 
$$\overline \G_S:=\cV(R)=\Delta_{E, \s}:=\{(p, \s(p))\subset \PP(V^*)\times \PP(V^*)\mid p\in E\},$$
$\widetilde E:=\Sigma(\overline \G_S)=\{\Sigma(p, \s(p))\in \PP(U^*)\mid p\in E\}$.    
Since 
$$\cV(W)=\{(p, \s(p), \s^2(p), \s^3(p))\in \PP(V^*)\times \PP(V^*)\times \PP(V^*)\times \PP(V^*)\mid p\in E\},$$
\begin{align*}
\cV(\widetilde W) & \supset (\Sigma\times \Sigma)(\cV(W)) \\
& =(\Sigma\times \Sigma)(\{(p, \s(p), \s^2(p), \s^3(p))\in \PP(V^*)^{\times 4}
\mid p\in E\} \\
& =\{(\Sigma (p, \s(p)), \Sigma (\s^2\times \s^2)(p, \s(p)))\in \PP(U^*)^{\times 2}
\mid p\in E\} \\
& =\{(\Sigma (p, \s(p)), \tilde \s\Sigma (p, \s(p))\in \PP(U^*)^{\times 2}
\mid p\in E\} \\
& =:\Delta _{\widetilde E, \tilde \s}
\end{align*}
by Lemma \ref{lem.tilw}.  By Lemma \ref{lem.2Ve},  
\begin{align*}
\Delta _{E^{(2)}, \s^{(2)}} & :=\{(p, \s^{(2)}(p))\mid p\in E^{(2)}\} =\{((\pi^*)^{-1}(q), \s^{(2)}(\pi^*)^{-1}(q))\mid q\in \widetilde E\} \\
& =\{((\pi^*)^{-1}(q), (\pi^*)^{-1}\tilde \s(q))\mid q\in \widetilde E\} =(\pi^*\times \pi^*)^{-1}(\Delta _{\widetilde E, \tilde \s}) \\
& \subset (\pi^*\times \pi^*)^{-1}(\cV(\widetilde W))=\overline \G_{S^{(2)}}.
\end{align*}  
\end{proof} 

\begin{lemma} \label{lem.q} 
Let $S$ be a quadratic algebra satisfying (G1) with $\cP(S)=(E, \s)$.  
If $S$ also satisfies (G2), then $\widetilde E\not \subset  \cV(\tilde f)$ for every $0\neq f\in S_2$.
\end{lemma}  

\begin{proof} If $\widetilde E\subset  \cV(\tilde f)$, then
$$\Sigma (\overline \G_S)=:\widetilde E=\widetilde E\cap \Im \Sigma\subset \cV(\tilde f)\cap \Im \Sigma=\Sigma (\cV(f))$$
by Lemma \ref{lem.tilw}.  Since $\Sigma:\PP^{n-1}\times \PP^{n-1}\to \PP^{n^2-1}$ is injective, $\overline \G_S\subset \cV(f)$, so $f(p, \s(p))=0$ for every $p\in E$.  Since $S$ satisfies (G2), $f=0$ in $S$, which is a contradiction. 
\end{proof} 

We have a rather short list of 3-dimensional Calabi-Yau quantum polynomial algebras $S$ such that $Z(S)_2\neq 0$.

\begin{theorem} \label{thm.cla} 
Let $S=\cA^+(E, \s)=k\<x, y, z\>/(f_1, f_2, f_3)$ be a 3-dimensional Calabi-Yau quantum polynomial algebra.  Then $Z(S)_2\neq 0$ if and only if $\bar \s^2=\id$ if and only if $S$ is isomorphic to one of the algebras in Table 2 below.  For each algebra, $Z(S)_2$ is also given in Table 2. 
\begin{center}{\rm Table $2$}\end{center}
$$\begin{array}{|c|c|c|c|}
\hline
f_1, f_2, f_3 & Z(S)_2 & |\bar \s| & \textnormal{Type} \\
\hline
yz-zy, zx-xz, xy-yx & S_2 & 1 & P \\
\hline 
yz-zy+x^2, zx-xz, xy-yx & kx^2 & 1 & TL \\
\hline
yz+zy, zx+xz, xy+yx & kx^2+ky^2+kz^2 & 2 & S \\
\hline
yz+zy+x^2, zx+xz, xy+yx & kx^2+ky^2+kz^2 & 2 & S' \\
\hline
yz+zy+x^2, zx+xz+y^2, xy+yx & kx^2+ky^2+kz^2 & 2 & NC \\
\hline
yz+zy+\l x^2, zx+xz+\l y^2, xy+yx+\l z^2 & kx^2+ky^2+kz^2 & 2 & EC \\
\hline
\end{array}$$
where $\l\in k$ such that $\l^3\neq 0, 1, -8$. 
\end{theorem}  

\begin{proof} We use the notations in Proposition \ref{prop.ve2}.  
If $E$ is reduced (so that $\overline \G=\G$ is reduced), then $S$ satisfies (G1) and (G2) with $\cP(S)=(E, \s)$ by Remark \ref{rem.g2g1}. 
If $0\neq f\in Z(S)_2$, then $0\neq \bar f\in Z(S^{(2)})_1$, so ${\s^{(2)}}|_{E^{(2)}\setminus \cV(\bar f)}=\id$ by Proposition \ref{prop.ve2} and Lemma \ref{lem.deg1}.  Since $\tilde f(\pi^*(p))=(\pi (\tilde f))(p)=\bar f(p)$
for $p\in \PP(\overline U)$,  
$\pi^*(E^{(2)}\setminus \cV(\bar f))=\widetilde E\setminus \cV(\tilde f)$, so $\tilde \s|_{\widetilde E\setminus \cV(\tilde f)}=\id$.
 
Since $\Sigma$ restricts to an isomorphism $\G\to \widetilde E$ (as varieties), there exists an irreducible component $\G_0$ of $\G$ such that $\widetilde E_0:=\Sigma(\G_0)\not \subset  \cV(\tilde f)$ by Lemma \ref{lem.q}.  Since projections $p_1, p_2:\PP^2\times \PP^2\to \PP^2$ restrict to isomorphisms $\G\to E$ (as varieties), $E_0:=p_1(\G_0)$ is an irreducible component of $E$ such that $\G_0=\{(p, \s(p))\mid p\in E_0\}$.

Since ${\tilde \s}|_{\widetilde E_0\setminus \cV(\tilde f)}=\id$, we have ${\tilde \s}|_{\widetilde E_0}=\id$, so 
$(\s|_{E_0})^2=\id$ by the commutative diagram
$$\begin{CD}
\G_0 @>(\s|_{E_0})^2\times (\s|_{\s(E_0)})^2>\cong > \G_0 \\
@V\Sigma V\cong V @V\cong V\Sigma V \\
\widetilde E_0 @>\tilde \s|_{\widetilde E_0}>\cong > \widetilde E_0.
\end{CD}$$ 

By Lemma \ref{lem.SP}, we can check that $(\s|_{E_0})^2=\id$ for some irreducible component $E_0$ of $E$ if and only if $\s^2=\id$ if and only if $S$ is isomorphic to one of the algebras of Type P, S, S', NC, EC in Table 2 above.    
On the other hand, if $S=k[x, y, z]$ (Type P), then clearly, $Z(S)_2=S_2\neq 0$, and if $S$ is one of the algebras of Type S, S', NC, EC in Table 2 above, then $Z(S)_2=kx^2+ky^2+kz^2\neq 0$ by \cite [Lemma 3.6]{H}.  

If $S$ is of Type TL, then we can show that $Z(S)_2\neq 0$ if and only if $S\cong k\<x, y, z\>/(yz-zy+x^2, zx-xz, xy-yx)$, and, in this case, we can show that $Z(S)_2=kx^2$ by direct calculations and $\bar \s=\id$ by Lemma \ref{lem.SP}.  

If $S$ is of Type WL, then we can show that $Z(S)_2=0$ by direct calculations and $|\bar \s|=\infty$ by Lemma \ref{lem.SP}.
\end{proof} 

\begin{example} We explain our geometric method by using a 2-dimensional quantum polynomial algebra
$S=k\<u, v\>/(vu-\l uv)=\cA^+(\PP^1, \s)$
where $\s(a, b)=(\l a, b)$.  It is easy to see that $Z(A)_2\neq 0$ if and only if $\l^2=2$ if and only if $\s^2=\id$.  By calculations, we can show that $\cP(S^{(2)})=(E^{(2)}, \s^{(2)})$ where 
\begin{align*}
& \widetilde E=\Sigma(\overline \G_S)=\Sigma (\{((a, b), (\l a, b))\mid (a, b)\in \PP^1\}) \\
& \hspace{.75in} =\{(\l a^2, ab, \l ab, b^2)\mid (a, b)\in \PP^1\}\cong \PP^1, \\ 
& \tilde \s(\a, \b, \c, \d)=(\l^4\a, \l^2\b, \l^2 \c, \d), \textnormal { and } \\
& E^{(2)}=(\pi^*)^{-1}(\widetilde E)=\{(\l a^2, ab, b^2)\mid (a, b)\in \PP^1\}\cong \PP^1, \\ 
& \s^{(2)}(\a, \b, \c)=(\l^4\a, \l^2\b, \c).
\end{align*}
If $0\neq f\in Z(S)_2$, then $0\neq \bar f\in Z(S^{(2)})_1$, so $\s^{(2)}|_{E^{(2)}\setminus \cV(\bar f)}=\id$.  Since $E^{(2)}\not \subset \cV(\bar f)$, 
we have $\s^{(2)}=\id$, so 
$\l^2=1$.  
\end{example} 

We write  
$$k_{-1}[x, y, z]:=k\<x, y, z\>/(yz+zy, zx+xz, xy+yx)=:S^{(0, 0, 0)}.$$
The following corollary describes all the possible homogeneous coordinate algebras of noncommutative conics. 

\begin{corollary} \label{cor.cgc} 
If $S=\cA^+(E, \s)$ is a 3-dimensional Calabi-Yau quantum polynomial algebra, and
$0\neq f\in Z(S)_2$, then $A=S/(f)$ is isomorphic to either 
\begin{enumerate}
\item{} $k[x, y, z]/(x^2), k[x, y, z]/(x^2+y^2), k[x, y, z]/(x^2+y^2+z^2)$, and, in this case, $|\s|\neq 2$ and $A^!\cong k_{-1}[x, y, z]/(G_1, G_2)$ for some $G_1, G_2\in k_{-1}[x, y, z]_2$ ($A^!$ is graded commutative), or 
\item{} $S^{(\a, \b, \c)}/(ax^2+by^2+cz^2)$
for some $\a, \b, \c\in k$ and $(a, b, c)\in \PP^2$ where 
$$S^{(\a, \b, \c)}:=k\<x, y, z\>/(yz+zy+\a x^2, zx+xz+\b y^2, xy+yx+\c z^2),$$ 
and, in this case, $|\s|=2$ and $A^!\cong k[x, y, z]/(G_1, G_2)$ for some $G_1, G_2\in k[x, y, z]_2$ ($A^!$ is commutative). 
\end{enumerate} 
\end{corollary} 

\begin{proof}  By Theorem \ref{thm.cla}, $S$ is isomorphic to one of the algebras 
in Table 2.  

If $S$ is of Type P, then $A$ is isomorphic to 
$$k[x, y, z]/(x^2), k[x, y, z]/(x^2+y^2), k[x, y, z]/(x^2+y^2+z^2)$$ 
by Sylvester's theorem.  If $S$ is of Type TL, then $A \cong k[x, y, z]/(x^2)$ by Table 2.  In either case, $|\s|\neq 2$ by Table 2, and we can show that $A^!$ is isomorphic to 
$$k_{-1}[x, y, z]/(y^2, z^2), k_{-1}[x, y, z]/(x^2-y^2, z^2), k_{-1}[x, y, z]/(x^2-y^2, x^2-z^2)$$
by direct calculations.  
 
If $S$ is of Type S, S', NC, EC, then 
$A\cong S^{(\a, \b, \c)}/(ax^2+by^2+cz^2)$ where $(a, b, c)\in \PP^2$ and $|\s|=2$ by Table 2. By \cite[Theorem 3.7]{H}, $A^!\cong k[x, y, z]/(G_1, G_2)$ for some $G_1, G_2\in k[x, y, z]_2$.  For example, if $a \neq 0$, then 
$$
A^! \cong k[x,y,z] / (a(y^2 - \beta xz) - b(x^2 - \alpha yz), a (z^2 -\gamma xy) - c(x^2-\alpha yz)).
$$
\end{proof}

\begin{corollary} \label{cor.HMM}  There are at most 9 isomorphism classes of homogeneous coordinate algebras $A$ of noncommutative conics.
\end{corollary} 

\begin{proof} 
If $A$ is commutative, then there are exactly 3 isomorphism classes for $A$ by Corollay \ref{cor.cgc} (1).  On the other hand, if $A$ is not commutative, then $A^!\cong k[x, y, z]/(G_1, G_2)$ where $G_1, G_2\in k[x, y, z]_2$ by Corollary \ref{cor.cgc} (2).  Since $A$ is Koszul, 
$$H_{A^!}(t)=1/H_A(-t)=(1+t)^3/(1-t^2)=(1-t^2)^2/(1-t)^3,$$  
so $G_1, G_2\in k[x, y, z]_2$ is a regular sequence, that is,  $A^!$ is a quadratic complete intersection.  As a consequence of the classical result of ``a pencil of conics", there are at most 6 isomorphism classes for $A^!$ (\cite[Section 11, Chapter XIII]{HP}).  Since $A\cong A'$ if and only if $A^!\cong {A'}^!$, there are at most 6 isomorphism classes for $A$.
\end{proof}


\section{Classifications of $E_A$ and $C(A)$} 


Let $A=S/(f)$ be a noncommutative conic where $S$ is a 3-dimensional Calabi-Yau quantum polynomial algebra and $0\neq f\in Z(S)_2$.  In this section, we will show that $A$ satisfies (G1) so that the point variety $E_A\subset \PP^2$ of $A$ is well-defined, and 
classify $E_A$ up to isomorphism. 
Recall in Introduction that there is a finite dimensional algebra $C(A)$ which plays an important role to study 
$A$. 
If $A=S/(f)$ is a noncommutative conic, then $C(A)$ is a 4-dimensional algebra by \cite[Lemma 5.1 (3)]{SV}.  In \cite {H}, we fail to compute $C(A)$ when $S$ is of Type EC.  In this section, we complete the classification of $C(A)$ by using the geometry of $E_A$. 
We first set some definitions and notations.

The Clifford deformation of a quadratic algebra was introduced by He and Ye \cite{HY1}.  

\begin{definition} 
Let $S=T(V)/(R)$ be a quadratic algebra where $R\subset V\otimes V$. A linear map $\theta : R \to k$ is called a {\it Clifford map} if 
$$(\theta \otimes 1- 1 \otimes \theta)(V \otimes R \cap R \otimes V) = 0.$$ 
We define the {\it Clifford deformation} of $S$ associated to $\theta$ by 
$$
S(\theta) : = T(V) / (f - \theta(f)\mid f \in R).
$$
\end{definition}

Let $S=T(V)/(R)$ be a quadratic algebra where $R\subset V\otimes V$.  Note that the $\ZZ_2$-graded structure on $T(V)=(\oplus _{i=0}^{\infty}V^{\otimes 2i})\oplus (\oplus _{i=0}^{\infty}V^{\otimes 2i+1})$ induces a $\ZZ_2$-graded structure on $S(\theta)$.  
The set of Clifford maps $\theta$ of $S^!=T(V^*)/(R^{\perp})$ is in one-to-one correspondence with the set of central elements $f \in Z(S)_2$ by \cite[Lemma 2.8]{HY1}.  We denote by $\theta_f$ the Clifford map of $S^!$ corresponding to $f\in Z(S)_2$.  


If $C$ is a finitely generated commutative algebra over $k$, and $G$ is a finite group acting on $C$, then $G$ naturally acts on  $\Spec C$, and there exists a bijection between the set of $G$-orbits of the closed points of $\Spec C$ and the set of closed points of $\Spec C^G$.
By this reason, we define $(\Spec C)/G:=\Spec C^G$.  

\begin{definition} 
Let $S$ be a connected graded algebra.  
For $f, f'\in S$, we write $f\sim f'$ if $f'=\l f$ for some $0\neq \l\in k$.  
For $f\in S_2$, we define 
$$K_f:=\{g\in S_1\mid g^2\sim f\}/\sim.$$
\end{definition} 

Note that $\#(K_f)$ is exactly the number of $g\in S_1$ up to sign such that $g^2=f$.  

\begin{proposition}  \label{prop.fgE} 
Let $S=\cA^+(E, \s)$ be a 3-dimensional Calabi-Yau quantum polynomial algebra such that $|\s|=2$, $0\neq f\in Z(S)_2$, and $A=S/(f)$. 
\begin{enumerate} 
\item{} For every $g\in S_1$, $g^2\in Z(S)_2$.  
\item{} $\#(K_f)=\#(\Spec C(A))=\#(\Proj A^!)=\#(E_{A^!})=1, 2, 3, 4$. 
\end{enumerate} 
\end{proposition} 

\begin{proof}  We may assume that 
$$S=k\<x, y, z\>/(yz+zy+\a x^2, zx+xz+\b y^2, xy+yx+\c z^2)$$ 
for some $\a, \b, \c\in k$, and 
$f=ax^2+by^2+cz^2$ for $a, b, c\in k$ by Corollary \ref{cor.cgc}. 

(1) For every $g=\l x+\mu y+\nu z\in S_1$,  
\begin{align*} 
g^2 &=  \l^2 x^2 + \mu^2 y^2 + \nu^2 z^2 + \l\mu(xy + yx) + \nu\l(xz + zx) + \mu\nu(yz + zy)\\
&= (\l^2-\a \mu\nu)x^2+(\mu^2-\b \nu\l)y^2+(\nu^2-\c \l\mu)z^2\in Z(S)_2
\end{align*}
by Theorem \ref{thm.cla}.  

(2) 
It is easy to calculate that the quadratic dual of $S$ is 
$$S^! \cong k[u,v,w]/(u^2-\a vw, v^2-\b wu, w^2-\c uv),$$ 
and the Clifford map $\theta_f$ of $S^!$ induced by $f$
is given by 
$$
u^2-\a vw\mapsto a, v^2-\b wu \mapsto b, w^2-\c uv \mapsto c,
$$
so the Clifford deformation of $S^!$ associated to $\theta_f$ is
$$S^!(\theta_f)\cong k[u, v, w]/(u^2-\a vw-a, v^2-\b wu-b, w^2-\c uv-c)$$ 
as $\ZZ_2$-graded algebras.
A finite group $G=\{\e, \d\}$ acts on $S^!(\theta_f)$ by $\d(h)=(-1)^{\deg h}h$ for a homogeneous element $h\in S^!(\theta_f)$ so that  
$(S^!(\theta_f))^{G}= (S^!(\theta_f))_0
\cong C(A)$ by \cite[Proposition 4.3]{HY1}.    

On the other hand, $g^2=f=ax^2+by^2+cz^2$ if and only if 
$$(\l, \mu, \nu)\in \cV(u^2-\a vw-a, v^2-\b wu-b, w^2-\c uv-c)=\overline{\Spec S^!(\theta_f)}$$  
by (1). 
Since the natural action of $G$ on $\overline{\Spec S^!(\theta_f)}$ is given by $\d^*(\l, \mu, \nu)=(-\l, -\mu, -\nu)$, we see that $\#(K_f)=\#((\Spec S^!(\theta_f))/G)$.   
Since $A^!$ and $C(A)$ are commutative by Corollary \ref{cor.cgc},
$$(\Spec S^!(\theta_f))/G= \Spec (S^!(\theta_f))^G\cong \Spec C(A)\cong \Proj A^!=E_{A^!}\neq \emptyset $$ 
by Lemma \ref{lem.g1c}. Since $\#(\Spec C(A))\leq \dim_k C(A)=4$, 
the result follows.  
\end{proof} 

\begin{lemma} \label{lem.g1a} Let $S$ be a quadratic algebra satisfying (G1) with $\cP(S)=(E, \s)$, $0\neq f\in S_2$, and $A=S/(f)$.  If $\s^2=\id$ and $f=g^2$ for some $g\in S_1$  (that is, $K_f\neq \emptyset$), then $A$ satisfies (G1) with $\cP(A)=(E_A, \s_A)$ where 
$$E_A=(E\cap \cV(g))\cup \s(E\cap \cV(g))$$ 
and $\s_A=\s|_{E_A}$.  
\end{lemma}  

\begin{proof} 
Since 
$$\overline \G_A=\overline \G\cap \cV(f)=\{(p, \s(p))\mid p\in E, f(p, \s(p))=g(p)g(\s(p))=0\},$$ 
$(p, \s(p))\in \overline \G_A$ if and only if $p\in E$ and $g(p)g(\s(p))=0$ if and only if $(\s(p), \s^2(p))=(\s(p), p)\in \overline \G_A$, so $A$ satisfies $(G_1)$ with $\cP(A)=(E_A, \s_A)$ where 
\begin{align*}
E_A & =\{p\in E\mid f(p, \s(p))=g(p)g(\s(p))=0\} \\
& =\{p\in E\mid g(p)=0\}\cup \{p\in E\mid g(\s(p))=0\} \\
& =(E\cap \cV(g))\cup (E\cap \s^{-1}(\cV(g))) \\
& =(E\cap \cV(g))\cup \s(E\cap \cV(g)),
\end{align*} 
and  $\s_A=\s|_{E_A}$. 
\end{proof} 

\begin{proposition} \label{prop.g1} 
If $S=\cA^+(E, \s)$ is a 3-dimensional Calabi-Yau quantum polynomial algebra, $0\neq f\in Z(S)_2$, and $A=S/(f)$, then both $A$ and $A^!$ satisfy (G1).   
\end{proposition} 

\begin{proof} 
The case $|\s|\neq 2$:  Since $A$ is commutative by Corollary \ref{cor.cgc}, $A$ satisfies (G1) by Lemma \ref{lem.g1c}.  Since $k_{-1}[x, y, z]=\cA(\mathcal{V}(xyz), \t)$ is a geometric quantum polynomial algebra with $|\t|=2$, and $A^!$ is isomorphic to one of the following algebras
\begin{align*}
& k_{-1}[x, y, z]/(y^2, z^2), \\
& k_{-1}[x, y, z]/((x+\sqrt{-1}y)^2, z^2), \\
& k_{-1}[x, y, z]/((x+\sqrt {-1}y)^2, (x+\sqrt {-1}z)^2)
\end{align*}
by the proof of Corollary \ref{cor.cgc},
$A^!$ satisfies (G1) by applying Lemma \ref{lem.g1a} twice.  

The case $|\s|=2$:  
By Proposition \ref{prop.fgE}, there exists $g\in S_1$ such that $f=g^2$, so $A$ satisfies (G1) by Lemma \ref{lem.g1a}.  Since $A^!$ is commutative by Corollary \ref{cor.cgc}, $A^!$ satisfies (G1) by Lemma \ref{lem.g1c}.  
\end{proof}

\begin{example} \label{exm-type-S-E}
By using Lemma \ref{lem.g1a}, we calculate the point variety $E_A=(E\cap \cV(g))\cup \s(E\cap \cV(g))$ for a noncommutative conic $A=S/(f)$ where $S$ is a 3-dimensional Calabi-Yau quantum polynomial algebra of Type S and $g\in S_1$ such that $g^2=f$.
By Corollary \ref{cor.cgc}, we may assume that $S=S^{(0, 0, 0)}=k_{-1}[x, y, z]$ and $f$ is one of $x^2, x^2+y^2, x^2+y^2+z^2\in Z(S)_2$.  
If $\cP(S)=(E, \s)$, then $E = \cV(xyz)=\cV(x)\cup \cV(y)\cup \cV(z)$ and $\left\{
             \begin{array}{ll}
              \sigma(0,b,c)=(0, b, -c), \\
              \sigma(a,0,c)=(-a,0, c), \\
              \sigma(a,b,0)=(a,-b, 0)
             \end{array}
             \right.
             $ by Lemma \ref{lem.SP}.  
\begin{enumerate}
\item[(1)] If $f = x^2$, then $E\cap \cV(g)=\cV(xyz)\cap \cV(x)=\cV(x)$.
Since $\s(\cV(x))=\cV(x)$,  
$
E_A = \cV(x)\subset \PP^2
$
is a line. 
\item[(2)] If $f = x^2 + y^2$, then $f = (x+y)^2$, so 
$$E\cap \cV(g)=\cV(xyz)\cap \cV(x+y)= \{ (1,-1,0),(0,0,1)\} \subset \PP^2.
$$
Since $\s(1,-1,0) = (1,1,0),\s(0,0,1) = (0,0,1)$, 
$$
E_A = \{ (1,-1,0),(0,0,1), (1,1,0)\} \subset \PP^2,
$$ 
so  $\#(E_A) = 3$.
\item[(3)] If $f = x^2 + y^2 + z^2$, then $f = (x+y+z)^2$, so
$$E\cap \cV(g)  = \cV(xyz)\cap \cV(x+y+z)=\{ (0, 1,-1),(-1,0,1),(1,-1, 0)\} \subset \PP^2.
$$
Since $\s(0,1,-1) = (0, 1,1), \s(-1, 0, 1)=(1, 0, 1), \s(1,-1,0) = (1,1,0)$, 
$$
E_A = \{ (0, 1,-1),(-1,0,1),(1,-1, 0),(0, 1,1),(1,0,1),(1,1, 0)\}\subset \PP^2,
$$ so  $\#(E_A) = 6$.
\end{enumerate}
\end{example}

Recall that every 3-dimensional Calabi-Yau quantum polynomial algebra of Type EC is isomorphic to a 3-dimensional Sklyanin algebra (see \cite[Corollary 4.3]{Ma}), and every 3-dimensional Sklyanin algebra $S$ such that $Z(S)_2\neq 0$ is of the form 
$$S=k\<x, y, z\>/(yz+zy+\l x^2, zx+xz+\l y^2, xy+yx+\l z^2)$$
where $\l^3\neq 0, 1, -8$ by Theorem \ref{thm.cla}.  In this case, $(E, \s)=\cP(S)=\cP^+(S)$ is described as follows: $E=\cV(F)\subset \PP^2$ is a smooth elliptic curve where 
$$F:=\l (x^3+y^3+z^3)-(\l^3+2)xyz\in k[x, y, z]_3.$$  
If we endow a group structure on $E$ with the identity $o:=(1, -1, 0)\in E$, then $\s\in \Aut E$ is a translation by a 2-torsion point $(1, 1, \l)\in E[2]$.  We use the fact that $\s(p)+\s(q)=p+q$ for every points $p, q\in E$.   More explicitly, 
$\s$ can be described as
$$\s(u, v, w)=\begin{cases} (\l v^2-uw, \l u^2-vw, w^2-\l ^2uv) & \textnormal { if } (u, v, w)\in E\setminus E_1  \\
(\l w^2-uv, v^2-\l^2 uw, \l u^2-vw) & \textnormal { if } (u, v, w)\in E\setminus E_2 \end{cases}$$
where $E_1:=\{(1, \e, \l \e^2)\mid \e^3=1\}, E_2:=\{(1, \l \e^2, \e)\mid \e^3=1\}$ (cf. \cite{Fr}).  

\begin{lemma} \label{lem.abc} 
Let $S=\cA^+(E, \s)$ be a 3-dimensional Sklyanin algebra, and $0\neq f=ax^2+by^2+cz^2\in Z(S)_2$.  If $a^3=b^3=c^3$, then $\#(K_f)=4$.  
\end{lemma}
  
\begin{proof} Without loss of generality, we may assume that $a=1$ and $b^3=c^3=1$. By the proof of Corollary \ref{cor.cgc},
$A^!\cong k[x, y, z]/(G_1, G_2)$ where 
$$G_1=(y^2-\l xz)-b(x^2-\l yz), G_2=(z^2-\l xy)-c(x^2-\l yz)\in k[x, y, z]_2.$$
By calculations,
$$\cV(G_{1x}, G_{1y}, G_{1z})=\cV(-\l z-2bx, 2y+b\l z, -\l x+b\l y)=\{(b^2\l, b\l, -2)\},$$
so $\cV(G_1)$ is a pair of lines meeting at $(b^2\l, b\l, -2)$.  If 
$$G_2(b^2\l, b\l, -2)=-(\l^3+3bc\l^2-4)=-(\l -bc)(\l+2bc)^2=0,$$
then $\l=bc$ or $\l=-2bc$, so $\l^3=b^3c^3=1$ or $\l^3=-8b^3c^3=-8$, which is a contradiction, hence $(b^2\l, b\l, -2)\not \in \cV(G_2)$.  By symmetry, $\cV(G_2)$ is a pair of lines meeting at $(c^2\l, -2, c\l)\not \in \cV(G_1)$, so $\#(K_f)=\#(\Proj A^!)=\#(\cV(G_1)\cap \cV(G_2))=4$ by Proposition \ref{prop.fgE}.
\end{proof} 
 
\begin{example} 
Let $S=\cA^+(E, \s)$ be a 3-dimensional Sklyanin algebra, and $0\neq f=ax^2+by^2+cz^2\in Z(S)_2$. 	If $a^3=b^3=c^3$, then $S/(f)\cong S/(x^2+y^2+\e z^2)$ where $\e^3=1$ by \cite[Lemma 4.4, Theorem 4.5]{H}.  If $g\in S_1$ is one of the following 
$$x+y+\e^2z, x+y-\e(\e+\l)z,  \e x-(\e+\l)y+z, -(\e+\l) x+\e y+z\in S_1,$$
then we can directly check that $g^2\sim f$, so $\#(K_f)\geq 4$.  By Proposition \ref{prop.fgE}, we have $\#(K_f)=4$. 		

 		We can find $g\in S_1$ such that $g^2\sim f$ using geometry as follows:  if 
 		$		
 		p = (1,0,-\e),q = (0,1,-\e) \in E		
 		$,  		
 		then $\sigma(o) = (1,1,\l), \sigma(p) = (\e,\l,\e^2), \sigma(q) = (\l,\e,\e^2)$, and $f(o,\s(o)) = f(p,\sigma(p)) = f(q,\sigma(q)) = 0$.  By Lemma \ref{lem.g1a} and by Bezout's theorem, we see that $\#(E_A) \leq 2\#(E\cap \cV(g))\leq 6$, so  	
 		$E_{A} = \{o,p,q, \s(o),\sigma(p),\sigma(q)\}$. 		
 		If $g_1,g_2,g_3,g_4 \in S_1$ such that 
 		\begin{align*} 		
 		& E\cap \mathcal{V}(g_1) = \{p,q\},E\cap \mathcal{V}(g_2) = \{o,\sigma(p)\}, \\
 		& E\cap \mathcal{V}(g_3) = \{p, \s(o) \},E\cap \mathcal{V}(g_4) = \{q, \s(o) \}, 		
 		\end{align*}		
 		then we can check that  
 		$$g_1 \sim x+y+\e^2z, g_2 \sim x+y-\e(\e+\l)z,  g_3 \sim \e x-(\e+\l)y+z, g_4 \sim -(\e+\l) x+\e y+z.$$
\end{example}

\begin{lemma} \label{lem.ng2} Let $S$ be a 3-dimensional Sklyanin algebra.  For every $0\neq f\in Z(S)_2$, $\#(K_f)\geq 2$.  
\end{lemma} 

\begin{proof} If $\#(K_f)= 1$, then $\#(\Spec C(A)) = \#(K_f) =1$ by Proposition \ref{prop.fgE}, so $C(A)$ is isomorphic to either $k[u,v]/(u^2,v^2)$ or $k[u]/(u^4)$ (cf. \cite[Corollary 4.10]{H}).   We will show that this is not the case.    

Let $0\neq f=ax^2+by^2+cz^2\in Z(S)_2$ and $A = S/(f)$.  By Lemma \ref{lem.abc}, we may assume that either $a^3\neq c^3$ or $b^3\neq c^3$.  Without loss of generality, we may assume that $c=1$.  If $a^3\neq 1$, then $C(A) \cong k[u]/(\varphi(u))$ where
\begin{align*}
\varphi(u) =& u^4 + \frac{3b\l^2}{(\l^3 -1)} u^3 + \frac{3b^2 \l^4 - a\l^3 - 2a}{(\l^3 -1)^2} u^2 \\
& +\frac{ \l^2 (b^3 \l^4  - 2 a b \l^3 +(1+ a^3) \l  -a b
  )}{(\l^3 -1)^3} u 
 + \frac{-a b^2\l^4 + (a^3 b + b)\l^2 -a^2 }{(\l^3 -1)^3}
 \end{align*}
 by \cite[Example 11]{H}.  Since $C(A)$ is generated by one element over $k$ as an algebra, $C(A) \ncong k[u,v]/(u^2,v^2)$. If $\varphi(u) = (u - \mu)^4$ for some $\mu \in k$, then, by comparing the coefficients in both sides, we have 
 \begin{align*}
 \left\{\begin{array}{rl}  -4\mu &= \frac{3b\l^2}{\l^3-1}, \\
 6\mu^2 &= \frac{3b^2\l^4- a \l^3 - 2a}{(\l^3 - 1)^2}, \\
 -4\mu^3 &= \frac{ \l^2 (b^3 \l^4  - 2 a b \l^3 +(1+ a^3) \l  -a b
  )}{(\l^3 -1)^3}, \\
  \mu^4 &= \frac{-a b^2\l^4 + (a^3 b + b)\l^2 -a^2}{(\l^3 -1)^3}.
  \end{array}\right.
 \end{align*}
By solving the above system of equations, we get 
\begin{align*}
& \l = a = \mu = 0, \; \;  \textnormal { or } \\
& \l^3 = -8, a^3=1, b= -a\l/2, \mu = a/3,
\end{align*} 
which contradicts the condition that $\lambda^3 \neq 0,1, -8$, so $C(A) \ncong k[u]/(u^4)$.

By symmetry, the result holds for the case that $b^3 \neq 1$.  
\end{proof} 

\begin{lemma} \label{lem.FH} Let $S=\cA^+(E, \s)$ be a 3-dimensional Sklyanin algebra, $0\neq f=ax^2+by^2+cz^2\in Z(S)_2$ and $A=S/(f)$.  Define 
\begin{align*}
& H_1=ax(\l y^2-xz)+by(\l x^2-yz)+cz(z^2-\l^2 xy), \\
& H_2=ax(\l z^2-xy)+by(y^2-\l^2 xz)+cz(\l x^2-yz)\in k[x, y, z]_3.
\end{align*}
Then the following holds:   
\begin{enumerate} 
\item{} $E\cap \cV(H_1)=E_A\cup E_1$.   If $a^3\neq b^3$, then $E\cap \cV(H_1)=E_A\sqcup E_1$ (disjoint union).
\item{} $E\cap \cV(H_2)=E_A\cup E_2$.   If $a^3\neq c^3$, then $E\cap \cV(H_2)=E_A\sqcup E_2$ (disjoint union).
\end{enumerate}
\end{lemma} 

\begin{proof} (1)  It is easy to check that $E_1:=\{(1, \e, \l \e^2)\mid \e^3=1\}\subset E\cap \cV(H_1)$, so $p\in E_A\cup E_1$ if and only if $p\in E\setminus E_1$ and $H_1(p)=0$, or $p\in E_1$ if and only if $p\in E$ and $H_1(p)=0$ if and only if $p\in E\cap \cV(H_1)$,  hence $E\cap \cV(H_1)=E_A\cup E_1$.   If $(1, \e, \l \e^2)\in E_A$, then 
$$f((1, \e, \l \e^2), \s(1, \e, \l \e^2))=f((1, \e, \l \e^2), (1, -\e, 0))=a-b\e^2=0,$$ 
so $a^3=b^3$, hence the result. 

(2) This holds by symmetry. 
\end{proof} 

The next result follows from Max Noether's Fundamental Theorem (\cite[Section 5.5]{F}).

\begin{lemma} \label{lem.Noe}  Let $0\neq F, G, H\in k[x, y, z]_3$.  If $\#(\cV(F, G))=9$ and $\cV(F, G)\subset \cV(H)$, then $H=\mu F+\nu G$ for some $\mu, \nu\in k$. 
\end{lemma}

Let $S=\cA^+(E, \s)$ be a 3-dimensional Calabi-Yau quantum polynomial algebra such that $|\s|=2$, and $0\neq f\in Z(S)_2$.  Recall that $g^2\in Z(S)_2$ for every $g\in S_1$ by Proposition \ref{prop.fgE} (1) so that the point variety $E_{S/(g^2)}\subset \PP^2$ of $S/(g^2)$ is well-defined by Proposition \ref{prop.g1}.  
We define 
$$L_f:=\{g\in S_1\mid E_{S/(g^2)}=E_{S/(f)}\}/\sim.$$  
If $g^2\sim f$, then $E_{S/(g^2)}=E_{S/(f)}$, so $K_f\subset L_f$, hence $\#(K_f)\leq \#(L_f)$. 

\begin{proposition} \label{prop.fAE} 
Let $S=\cA^+(E, \s)$ be a 3-dimensional Sklyanin algebra, $0\neq f=ax^2+by^2+cz^2\in Z(S)_2$ such that $a^3\neq b^3$ or $a^3\neq c^3$, and $A=S/(f)$.   If $\#(E_A)=6$,  
then $K_f=L_f$ 
so that $\#(K_f)=\#(L_f)$.  
\end{proposition}

\begin{proof} It is enough to show that $L_f\subset K_f$.  Suppose that $a^3\neq b^3$.  Since $\#(E_A)=6$,
$$\#(\cV(F, H_1))=\#(E\cap \cV(H_1))=\#(E_A\sqcup E_1)=\#(E_A)+\#(E_1)=9$$ 
where $F=\l (x^3+y^3+z^3)-(\l^3+2)xyz\in k[x, y, z]_3$ and
$$H_1=ax(\l y^2-xz)+by(\l x^2-yz)+cz(z^2-\l^2 xy)\in k[x, y, z]_3$$ 
by Lemma \ref{lem.FH}.  Let $g\in L_f$, $g^2=a'x^2+b'y^2+c'z^2\in Z(S)_2$, and 
$$H'_1=a'x(\l y^2-xz)+b'y(\l x^2-yz)+c'z(z^2-\l^2 xy)\in k[x, y, z]_3.$$  
Since $E_{S/(g^2)}=E_{S/(f)}=E_A$, 
$$\cV(F, H_1)=E\cap \cV(H_1)=E_A\cup E_1=E_{S/(g^2)}\cup E_1=E\cap \cV(H'_1)\subset \cV(H'_1)$$
by Lemma \ref{lem.FH}, so $H'_1=\mu F+\nu H_1$ for some $\mu, \nu\in k$ by Lemma \ref{lem.Noe}.  Comparing the coefficients of $x^3$ in both sides, we can see that $\mu=0$, so $H'_1=\nu H_1$.  It follows that $g^2=\nu f\sim f$, so $g\in K_f$.  

By symmetry, the result holds for the case that $a^3\neq c^3$. 
\end{proof} 

\begin{theorem} \label{thm.Skl} If $S$ is a 3-dimensional Sklyanin algebra, $0\neq f\in Z(S)_2$, and $A=S/(f)$, then exactly one of the following holds:   
\begin{enumerate}
\item{} $\#(E_A)=2$ and $\#(E_{A^!})=2$. 
\item{} $\#(E_A)=4$ and $\#(E_{A^!})=3$. 
\item{} $\#(E_A)=6$ and $\#(E_{A^!})=4$. 
\end{enumerate}
\end{theorem} 

\begin{proof} By Lemma \ref{lem.g1a}, and Proposition \ref{prop.g1}, $E_A=(E\cap \cV(g))\cup \s(E\cap \cV(g))$ for some $g\in S_1$.  Since $\s(p)\neq p$ for every $p\in E$, either \begin{enumerate}
\item{} $E_A=\{p, \s(p)\}$, 
\item{} $E_A=\{p, q, \s(p), \s(q)\}$, or 
\item{} $E_A=\{p, q, r, \s(p), \s(q), \s(r)\}$
\end{enumerate} 
for some distinct points $p, q, r\in E$.  Since $\#(E_{A^!})=\#(K_f)$ by Proposition \ref{prop.fgE}, we will compute $\#(K_f)$ in each case. 
We use the following geometric properties. 
\begin{itemize} 
\item{} For every $0\neq g\in S_1$, $0\neq g^2\in Z(S)_2$ by Proposition \ref{prop.fgE} (1), so $\#(K_{g^2})\geq 2$  by Lemma \ref{lem.ng2}. 
\item{} For $g\in L_f$, $E_{S/(g^2)}=(E\cap \cV(g))\cup \s(E\cap \cV(g))=E_A$ by Lemma \ref{lem.g1a}, so 
$$\#(E_A)/2\leq \#(E\cap \cV(g))\leq \max\{\#(E_A), 3\}$$
by Bezout's Theorem.  
\item{} For $0\neq g\in S_1$, $\#(E\cap \cV(g))=1, 2$ if and only if $\cV(g)$ is a tangent line at some point of $E$.  
\item{} For $0\neq g, g'\in S_1$, $g\sim g'$ if and only if $\cV(g)=\cV(g')$ if and only if $E\cap \cV(g)=E\cap \cV(g')$.  
\end{itemize}

(1) Suppose that $E_A=\{p, \s(p)\}$ so that $\#(E\cap \cV(g))=1, 2$ for every $g\in L_f$.  Since $\cV(g)$ is a tangent line at $p$ or at $\s(p)$ (but not at both points), $\#(K_f)\leq 2$.  
By Lemma \ref{lem.ng2}, $\#(K_f)=2$.    

(2) Suppose that $E_A=\{p, q, \s(p), \s(q)\}$ so that $\#(E\cap \cV(g))=2, 3$ for every $g\in L_f$.   By geometry, it is easy to see that 
there exists $g_1\in L_f$ such that $\#(E\cap \cV(g_1))=2$ so that $\cV(g_1)$ is a tangent line at some point of $E$.  Without loss of generality, we may assume that $\cV(g_1)$ is a tangent line at $p$ and $E\cap \cV(g_1)=\{p, q\}$ so that $2p+q=0$.
Since $2\s(p)+q=2p+q=0$, there exists $g_2\in L_f$ such that $E\cap \cV(g_2)=\{\s(p), q\}$ ($\cV(g_2)$ is a tangent line at $\s(p)$).  Since $p+\s(p)+\s(q)=2p+q=0$, there exists $g_3\in L_f$ such that $E\cap \cV(g)=\{p, \s(p), \s(q)\}$.  By  geometry, it is easy to see that 
there is no other possible $g\in L_f$,
so $\#(K_f)\leq \#(L_f)=3$.
By Lemma \ref{lem.ng2}, $\#(K_{g_i^2})\geq 2$ for $i=1, 2, 3$, so, for every $i$, there exists $j\neq i$ such that $g_i^2\sim g_j^2$. 
It follows that $g_1^2\sim g_2^2\sim g_3^2\sim f$, so $\#(K_f)=3$. 

(3) Let $f=ax^2+by^2+cz^2\in Z(S)_3$.  If $a^3=b^3=c^3$, then $\#(K_f)=4$ by Lemma \ref{lem.abc}, so we will assume that either $a^3\neq b^3$ or $a^3\neq c^3$.  Suppose that $E_A=\{p, q, r, \s(p), \s(q), \s(r)\}$ so that $\#(E\cap \cV(g))=3$ for every $g\in L_f$.   Without loss of generality, we may assume that there exists $g_1\in L_f$ such that $E\cap \cV(g_1)=\{p, q, r\}$ so that $p+q+r=0$.   Since 
$$p+\s(p)+\s(q)=\s(p)+q+\s(r)=\s(p)+\s(q)+r=0,$$ 
there exist $g_2, g_3, g_4\in L_f$ such that 
\begin{align*}
& E\cap \cV(g_2)=\{p, \s(q), \s(r)\}, \\
& E\cap \cV(g_3)=\{\s(p), q, \s(r)\}, \\
& E\cap \cV(g_4)=\{\s(p), \s(q), r\},
\end{align*}
so $\#(K_f)=\#(L_f)\geq 4$   
by Proposition \ref{prop.fAE}.
By Proposition \ref{prop.fgE},  $\#(K_f)=4$. 
\end{proof}  

\begin{theorem} \label{thm.main2}  
Let $S=\cA^+(E, \s), S'=\cA^+(E', \s')$ be 3-dimensional Calabi-Yau quantum polynomial algebras, $0\neq f\in Z(S)_2, 0\neq f'\in Z(S')_2$, and $A=S/(f), A'=S'/(f')$.  
\begin{enumerate}
\item{} If $|\s|, |\s'|\neq 2$, then the following are equivalent: 
\begin{enumerate}
\item{} $A\cong A'$. 
\item{} $\grmod A\cong \grmod A'$. 
\item{} $E_A\cong E_{A'}$ (as schemes). 
\item{} $E_{A^!}\cong E_{{A'}^!}$ (as varieties). 
\item{} $C(A)\cong C(A')$.
\end{enumerate} Moreover, we have the following table:
$$\begin{array}{|c|c|c|c|}
\multicolumn{4}{c}{\textnormal{Table 3}} \\[5pt] \hline
E_A & \#(E_{A^!}) & C(A) & \textnormal { Type of $S$} \\
\hline
\textnormal{ a double line }  &  \textnormal{ 1 } & k_{-1}[u, v]/(u^2, v^2) & P, TL  \\
\hline
\textnormal{ two lines } &  \textnormal{ 2  } & k_{-1}[u, v]/(u^2, v^2-1) & P  \\
\hline 
\textnormal{ a smooth conic } & \textnormal{ 0  } & M_2(k)\cong k_{-1}[u, v]/(u^2-1, v^2-1) & P \\ 
\hline
\end{array}$$
\item{} If $|\s|=|\s'|=2$, then the following are equivalent: 
\begin{enumerate}
\item{} $E_A\cong E_{A'}$ (as varieties). 
\item{} $E_{A^!}\cong E_{{A'}^!}$ (as schemes).
\item{} $C(A)\cong C(A')$.
\end{enumerate} Moreover, we have the following table: 
$$\begin{array}{|c|c|c|c|}
\multicolumn{4}{c}{\textnormal{Table 4}} \\[5pt] \hline
\#(E_A) & \#(E_{A^!}) & C(A) & \textnormal {Type  of  $S$} \\
\hline
\infty  & \textnormal{ 1  } & k[u, v]/(u^2, v^2) & S, S'   \\
\hline
\textnormal{ 1  }  &  \textnormal{ 1  }   & k[u]/(u^4) & S', NC \\
\hline
\textnormal{ 2  }  &  \textnormal{ 2  }   & k[u]/(u^3)\times k & NC, EC \\
\hline
\textnormal{ 3  }  &  \textnormal{ 2  }   & 
(k[u]/(u^2))^2\cong k[u, v]/(u^2, v^2-1) & S, S', NC  \\
\hline
\textnormal{ 4  }  &  \textnormal{ 3  }   & k[u]/(u^2)\times k^2 & S', NC, EC  \\
\hline
\textnormal{ 6  }  &  \textnormal{ 4  }   & k^4\cong k[u, v]/(u^2-1, v^2-1) & S, S', NC, EC \\
\hline
\end{array}$$
In Table 4, if $\#(E_A)=\infty$, then $E_A$ is a line.
\end{enumerate}
\end{theorem} 

\begin{proof} 
(1) If $|\s|, |\s'|\neq 2$, then $A, A'$ are commutative conics by Corollary \ref{cor.cgc} so that $E_A=\Proj A$ by Lemma \ref{lem.g1c}, so it is well-known that (a) $\Leftrightarrow$  (b) $\Leftrightarrow$  (c).
It is easy to show Table 3 by direct calculations.     

(2) If $|\s|=2$, then $A^!$ and $C(A)$ are commutative by Corollary \ref{cor.cgc}, so $\Spec C(A)\cong \Proj A^!=E_{A^!}$ by Lemma \ref{lem.g1c}.  It follows that $E_{A^!}\cong E_{{A'}^!}$ (as schemes) if and only if $C(A)\cong C(A')$.  
  
We can show Table 4 by direct calculations except for Type EC (see Example \ref{exm-type-S-E} for the calculations of $E_A$, and see \cite[Theorem 4.8]{H} for calculations of $C(A)$).  Suppose that $S$ is of Type EC.  If $\#(E_A) = 4,6$, then Table 4 follows from Proposition \ref{prop.fgE}, Theorem \ref{thm.Skl} and \cite[Corollary 4.10]{H}. 
If $\#(E_A) = 2$, then, 
by the proof of Theorem \ref{thm.Skl}, we may assume that $E_A = \{p, \sigma(p)\}$ and there exists $g \in S_1$ such that $g^2\sim f$ and $\mathcal{V}(g)$ is a tangent line at $p$ to $E$ not passing through $\s(p)$ so that $p\in E[3]$.  
Since $p \in E[3]=\{(0,1, -\e),(-\e,0,1), (1, -\e,0)\mid \e^3 = 1\}$, $g$ is one of the following
$$
\e(\l^3 + 2) x + 3\l y + 3 \e^2 \l z, 3\e^2\l x + \e(\l^3 + 2) y + 3  \l z, 3\l x + 3 \e^2 \l y + \e(\l^3 + 2)z,  
$$
so $g^2$ associates to one of the following
$$
\frac{(4-\lambda^3)}{3\lambda^2} x^2 +  \e y^2 + \e^2 z^2, \e^2 x^2 +  \frac{(4-\lambda^3)}{3\lambda^2} y^2 + \e z^2, \e x^2 +  \e^2 y^2 +\frac{(4-\lambda^3)}{3\lambda^2} z^2 
$$
by direct calculations.  By \cite[Lemma 4.4, Theorem 4.5]{H},
$$A=S/(f)=S/(g^2) \cong S/((4-\lambda^3)x^2+ 3\lambda^2y^2 + 3\lambda^2z^2).$$ 
Since $\l^3\neq 0, 1, -8$, we can show that $((4-\lambda^3)/3\lambda^2)^3 \neq 1$, so $C(A) \cong  k[u]/(\varphi(u))$ where 
\begin{align*}
 \varphi(u) =& u^4 + \frac{3\l^2}{(\l^3 -1)} u^3 + \frac{2(5\l^3+4)}{3\l^2(\l^3 -1)} u^2 +\frac{4(11\l^3+16)}{27\l^3(\l^3-1)} u + \frac{8(\l^3+2)}{27\l^4(\l^3 -1)} \\
 =& \left(u+\frac{2}{3\l}\right)^3\left(u+\frac{\l^3 + 2}{\l(\l^3-1)}\right)
\end{align*}
by \cite[Example 11]{H}.  Since $\l^3\neq 0, 1, -8$, we can show that $\frac{2}{3\l}\neq \frac{\l^3 + 2}{\l(\l^3-1)}$, so $C(A) \cong k[u]/(u^3) \times k$.
\end{proof}

Let $S=\cA^+(E, \s)$ and $A=S/(f)$ be as in Theorem \ref{thm.main2} above.  We describe $(E_A, \s_A)$ together with $\cup _{g\in K_f}\cV(g)$ in the picture below for the case that $\#(E_A)<\infty$.  In the picture, $E_A$ consists of hollow points and solid points where hollow points indicate singular points of $E$, and $\s_A$ is described by arrows.  Moreover, $\cV(g)$ for $g\in K_f$ are described by dotted lines.  Recall that $\#(K_f)=\#(E_{A^!})$ by Proposition \ref{prop.fgE}.

\begin{table}[h] 
\begin{center}
\begin{tabular}{ccccccccc}  
{\small$\#(E_A) = 1$} & & {\small$\#(E_A) = 2$} && {\small$\#(E_A) = 3$} && {\small$\#(E_A) = 4$} && {\small$\#(E_A) = 6$}  \vspace{+1mm}  \\ 

\begin{tikzpicture}[x=10, y=10]
\draw[white] (0, -2) -- (0, 2) ;
\draw[dashed] (-2, 0) -- (2, 0) ; 
\draw[->] (0,0) arc (90:400:1.5mm);

\draw [fill=white] (0,0) circle[radius= 0.13 em]; 
    
     
\end{tikzpicture}
& &
\begin{tikzpicture}[x=10, y=10]
\draw[dashed] (-1, 0) -- (3, 0) ; 
\draw[dashed] (0, -2) -- (0, 2) ;
\fill (0, 0) circle (1.5pt) ;
\fill (1.5, 0) circle (1.5pt) ;
     \draw[<->]
      (0,-0.15)
      to[out=-75, in=-120] (1.5,-0.15);
\end{tikzpicture} 

&&

\begin{tikzpicture}[x=10, y=10]
\draw[white] (0, -2) -- (0, 2) ;
\draw[dashed] (-1.5, -0.5) -- (3, 1) ;
\draw[dashed] (-1.5, 0.5) -- (3, -1) ;

\fill (2, 2/3) circle (1.5pt) ;
\fill (2, -2/3) circle (1.5pt) ;
\draw[->] (0,0) arc (90:400:1.5mm);
\draw[<->]
      (2.15,2/3)
      to[out=-10, in=30] (2.15,-2/3);
\draw [fill=white] (0,0) circle[radius= 0.13 em]; 
\end{tikzpicture} 
&&
\begin{tikzpicture}[x=10, y=10]
\draw[white] (0, -2) -- (0, 1.2) ;
\draw[dashed] (-1.5, 0) -- (3.7, 0) ; 
\draw[dashed] (-1.4, -0.8) -- (0.35,2.7) ; 
\draw[dashed] (-0.45, 2.6) -- (9/4, -1) ;
\fill (-1, 0) circle (1.5pt) ;
\fill (0, 2) circle (1.5pt) ;
\fill (3/2, 0) circle (1.5pt) ;
\fill (3, 0) circle (1.5pt) ;
\draw[<->]
      (-1+0.08, -0.15-0.02)
      to[out=-30, in=-150] (3/2-0.1, -0.15);
      \draw[<->]
      (0.15+0.03, 2)
      to[out=-13, in=130] (3-0.03, 0.13+0.05);
\end{tikzpicture}
&&
\begin{tikzpicture}[x=10, y=10]
\draw[white] (0, -2) -- (0, 1.2) ;
\draw[dashed] (-3.5, 0) -- (0.8, 0) ; 
\draw[dashed] (-2.6, -3.9) -- (0.4, 3/5) ; 
\draw[dashed] (-1.4, 0.6) -- (-2.2, -4.2) ; 
\draw[dashed] (-3.5, 0.375) -- (-0.2, -2.1) ; 
\fill (0, 0) circle (1.5pt) ; 
\fill (-3, 0) circle (1.5pt) ;
\fill (-1.5, 0) circle (1.5pt) ;
\fill (-1, -3/2) circle (1.5pt) ;
\fill (-2, -3) circle (1.5pt) ;
\fill (-5/3, -1) circle (1.5pt) ;
\draw[<->]
      (-0.15, -0.11)
      to[out=-150, in=30] (-5/3+0.16, -0.92);
\draw[<->]
      (-3.08, -0.18)
      to[out=-110, in=150] (-2.17, -2.9);
\draw[<->]
      (-1.4, -0.15)
      to[out=-60, in=100] (-1, -1.32);
\end{tikzpicture}
\end{tabular}
\end{center}
\end{table}

\section{Classification of $\Projn A$.}  

In this section, we classify noncommutative conics $\Projn A$ up to isomorphism of noncommutative schemes.


\begin{definition} [\cite {AZ}] A {\it noncommutative scheme} is a pair $X=(\mod X, \cO_X)$ consisting of an abelian category $\mod X$ and an object $\cO_X\in \mod X$.  We say that two noncommutative schemes $X, Y$ are isomorphic, denoted by $X\cong Y$, if there exists an equivalence functor $F:\mod X\to \mod Y$ such that $F(\cO_X)\cong \cO_Y$.
	
	For a right noetherian ring $R$, we define the {\it noncommutative affine scheme} associated to $R$ by $\Specn R=(\mod R, R)$ where $\mod R$ is the category of finitely generated right $R$-modules.  
	
	For a right noetherian connected graded algebra $A$, we define the {\it noncommutative projective scheme} associated to $A$ by $\Projn A=(\tails A, \cA)$ where $\tors A$ is the full subcategory of $\grmod A$ consisting of finite dimensional modules over $k$ and $\tails A:=\grmod A/\tors A$ is the quotient category.  Here, we denote by $\cM\in \tails A$ when we view $M\in \grmod A$ as an object in $\tails A$.
\end{definition}  

Let $A$ be a right noetherian connected graded algebra.   For $\mathcal{M}, \mathcal{N} \in \tails A$, we write $\Ext_{\cA}^i(\cM, \cN):=\Ext^{i}_{\tails A}(\mathcal{M},\mathcal{N})$
for the extension groups in $\tails A$.

\begin{lemma} \label{lem.spr} 
	For right noetherian rings $R, R'$, $\Specn R\cong \Specn R'$ if and only if $R\cong R'$. 
\end{lemma} 

\begin{proof} If $\Specn R\cong \Specn R'$, then there exists an equivalence functor $F:\mod R\to \mod R'$ such that $F(R)\cong R'$, so 
	$$R'\cong \End_{R'}(R')\cong \Hom_{R'}(F(R), F(R))\cong \Hom_R(R, R)\cong \End_R(R)\cong R.$$
	The converse is clear.  
\end{proof} 

Let $A, A'$ be noncommutative quadric hypersurfaces.  Despite the notation, it is not trivial to see if $A\cong A'$ implies $C(A)\cong C(A')$ from the definition.  We will show it using the notion of noncommutative schemes.  

\begin{lemma} \label{lem.cpn} 
	Let $A, A'$ be noncommutative quadric hypersurfaces.  
	\begin{enumerate}
		\item{} $\Projn A^!\cong \Projn {A'}^!$ if and only if $C(A)\cong C(A')$.
		\item{} If $\grmod A\cong \grmod A'$, then $C(A)\cong C(A')$.
	\end{enumerate}  
\end{lemma} 

\begin{proof} (1) Since there exists an equivalence functor $F:\tails A^!\to \mod C(A)$ defined by $F(\cM)=M[(f^!)^{-1}]_0$ by \cite[Lemma 4,13]{MU2}, $F(\cA)\cong C(A)$, so $\Projn A^!\cong \Specn C(A)$.  It follows that $\Projn A^!\cong \Projn {A'}^!$ if and only if $\Specn C(A)\cong \Specn C(A')$ if and only if $C(A)\cong C(A')$ by Lemma \ref{lem.spr}.
	
	(2) By \cite[Lemma 4.1]{MU1} and (1), 
	\begin{align*}
	\grmod A\cong \grmod A' & \Leftrightarrow \grmod A^!\cong \grmod {A'}^! \\
	& \Rightarrow \Projn A^!\cong \Projn {A'}^! \\
	& \Leftrightarrow C(A)\cong C(A').  
	\end{align*}
\end{proof}

\begin{definition}  A connected graded algebra $A$ is called an {\it AS-Gorenstein algebra of dimension $d$ and of Gorenstein parameter $l$}  if 
\begin{enumerate}
\item{} $\operatorname{id} A=d$, and
\item{} $\Ext^i_A(k, A(-j))\cong \begin{cases}  k & \textnormal { if $i=d$ and $j=l$, } \\
0 & \textnormal { otherwise. } \end{cases}$ 
\end{enumerate}
\end{definition} 

\begin{lemma} \label{lem.asg} If $S$ is a quantum polynomial algebra of dimension $d\geq 1$, and $0\neq f\in S_m$ is a normal element, then $A=S/(f)$ is a noetherian AS-Gorenstein algebra of dimension $d-1$ and of Gorenstein parameter $d-m$ (cf. \cite[Lemma 1.2]{HY1}). 
\end{lemma}  


\begin{definition} 
	Let $A$ be a right noetherian AS-Gorenstein algebra.  
	We say that $M \in \grmod A$ is {\it maximal Cohen-Macaulay}
	if $\Ext^{i}_{A}(M,A(j))=0$ for all $i \neq 0$ and all $j\in \ZZ$.
	We denote by 
	\begin{enumerate}
		\item{} $\CM^{\ZZ}(A)$ the full subcategory of $\grmod A$ consisting of maximal Cohen-Macaulay graded right $A$-modules, 
		\item{} $\CM^0(A):=\{M\in \CM^{\ZZ}(A)\mid M=M_0A\}$ the full subcategory of $\CM^{\ZZ}(A)$, and    
		\item{} $\uCM^{\ZZ}(A)$ the stable category of $\CM^{\ZZ}(A)$ by projectives. 
	\end{enumerate}
\end{definition} 

For a noncommutative conic $A$, we have a complete classification of $C(A)$ up to isomorphism by Theorem \ref{thm.main2}, so 
the following theorem gives a complete classification of $\uCM^{\ZZ}(A)$ up to equivalence. 

\begin{theorem} \label{thm.mcc}
Let $S, S'$ be 3-dimensional Calabi-Yau quantum polynomial algebras, $0\neq f\in Z(S)_2, 0\neq f'\in Z(S')_2$, and $A=S/(f), A'=S'/(f')$.
Then the following are equivalent: 
\begin{enumerate}
\item{} $\uCM^{\ZZ}(A)\cong \uCM^{\ZZ}(A')$. 
\item{} $\sD^b(\mod C(A))\cong \sD^b(\mod C(A'))$.  
\item{} $C(A)\cong C(A')$.  
\end{enumerate}
\end{theorem} 

\begin{proof} By Theorem \ref{thm.SV}, (1) $\Leftrightarrow$ (2), and clearly, (3) $\Rightarrow $ (2), so it is enough to show  (2) $\Rightarrow $ (3).   Suppose that $\sD^b(\mod C(A))\cong \sD^b(\mod C(A'))$.  By \cite[Proposition 9.2]{R}, 
$Z(C(A))\cong Z(C(A'))$.  
Moreover, $C(A)$ is semisimple if and only if $C(A')$ is semisimple, 
so we may assume one of the following two cases, namely, both $C(A)$ and $C(A')$ are semisimple ($M_2(k)$ or $k^4$), or both $C(A)$ and $C(A')$ are non semisimple.  By the classification of Theorem \ref{thm.main2}, we can show  in either case that $Z(C(A))\cong Z(C(A'))$ implies $C(A)\cong C(A')$
(Most of $C(A)$ are commutative, so it is enough to calculate $Z(k_{-1}[u, v]/(u^2, v^2))\cong k[t]/(t^2)$ and $Z(k_{-1}[u, v]/(u^2, v^2-1))\cong k$!!), so the result follows. 
\end{proof}

\begin{definition}
	Let $\mathscr{T}$ be a triangulated category.  A strictly full triangulated
	subcategory $\mathscr{E}$ of $\mathscr{T}$ is called {\it admissible} if the embedding functor $\mathscr{E}\to \mathscr{T}$ has both a left adjoint functor $\mathscr{T}\to \mathscr{E}$ and a right adjoint functor $\mathscr{T}\to \mathscr{E}$. 
	
	A {\it semi-orthogonal decomposition} of $\mathscr{T}$ is a sequence
	$\{ \mathscr{E}_{0}, \cdots, \mathscr{E}_{l-1} \}$ of strictly full admissible triangulated
	subcategories such that
	\begin{enumerate}
		\item{} for all $0 \leq i < j \leq l-1$ and all objects $E_{i} \in \mathscr{E}_{i}$,
		$E_{j} \in \mathscr{E}_{j}$, one has $\Hom_{\mathscr{T}}(E_{j},E_{i})=0$, and
		\item{} the smallest strictly full triangulated subcategory of $\mathscr{T}$
		containing $\mathscr{E}_{0}, \cdots, \mathscr{E}_{l-1}$ coincides with $\mathscr{T}$.
	\end{enumerate}
\end{definition}

We use the notation $\mathscr{T}=\<\mathscr{E}_{0}, \cdots, \mathscr{E}_{l-1}\>$
for a semi-orthogonal decomposition of $\mathscr{T}$ with components
$\mathscr{E}_{0}, \cdots, \mathscr{E}_{l-1}$.
		
\begin{theorem} [{\cite[Theorem 4.4.1]{B},\cite[Theorem 2.5]{O}}] \label{thm.Or}
Let $A$ be a right noetherian AS-Gorenstein algebra of Gorenstein parameter $l$.
If $l>0$, then there exists a fully faithful functor
$\Phi: \uCM^{\ZZ}(A) \to \sD^{b}(\tails A)$ and a semi-orthogonal decomposition
$$
\sD^{b}(\tails A)=\< \cA(-l+1), \cdots, \cA(-1), \cA, \Phi \uCM^{\ZZ}(A) \>
$$
where $\cA(i)$ denotes the full triangulated subcategory generated by the object
$\cA(i)$ by abuse of notation.
\end{theorem}

\begin{theorem} \label{thm.GA} 
Let $S, S'$ be 3-dimensional Calabi-Yau quantum polynomial algebras, $0\neq f\in Z(S)_2, 0\neq f'\in Z(S')_2$, and $A=S/(f), A'=S'/(f')$.  If $\Projn A\cong \Projn A'$, then $C(A)\cong C({A'})$. 
\end{theorem} 

\begin{proof}
Since $A$ is a noetherian AS-Gorenstein algebra of Gorenstein parameter 1 by Lemma \ref{lem.asg}, there exists a fully faithful functor $\Phi:\uCM^{\ZZ}(A)\to \sD^b(\tails A)$ such that $\sD^b(\tails A)=\<\cA, \Phi \uCM^{\ZZ}(A)\>$ is a semi-orthogonal decomposition by Theorem \ref{thm.Or}, so  
$$\uCM^{\ZZ}(A)\cong {^{\perp}\<\cA\>}:=\{\cX\in \sD^b(\tails A)\mid \Hom_{\cA}(\cX, \cA)=0\}$$  
by \cite [Lemma 1.4]{O}.  Since $\Projn A\cong \Projn A'$, there exists an equivalence functor $F:\sD^b(\tails A)\to \sD^b(\tails A')$ such that $F(\cA)\cong \cA'$.  It follows that   
$$\sD^b(\mod C(A))\cong \uCM^{\ZZ}(A)\cong {^{\perp}\<\cA\>}\cong {^{\perp}\<\cA'\>}\cong \uCM^{\ZZ}(A')\cong \sD^b(\mod C(A')),$$
so $C(A)\cong C(A')$ by Theorem \ref{thm.mcc}.
\end{proof}

\begin{theorem} \label{thm.HMM} There are exactly 9 isomorphism classes of homogeneous coordinate algebras $A$ of noncommutative conics, and there are exactly 9 isomorphism classes of noncommutative conics $\Projn A$.  
\end{theorem} 

\begin{proof} 
If $A$ and $A'$ are homogeneous coordinate algebras of noncommutative conics, then 
$$A\cong A'\Rightarrow\Projn A\cong \Projn A'\Rightarrow C(A)\cong C(A')$$ 
by Theorem \ref{thm.GA}.   Since there are at most 9 isomorphism classes of $A$ by Corollary \ref{cor.HMM}, and there are exactly 9 isomorphism classes of $C(A)$ 
by Theorem \ref{thm.main2}, 
the result follows. 
\end{proof}


 For the rest of the paper, we study ``smooth'' noncommutative conics.  The {\it global dimension} of $\tails A$ is defined by
$$
\gldim(\tails A):=\sup \{
i \mid \Ext^{i}_{\cA}(\mathcal{M}, \mathcal{N}) \neq 0 \,\,\,{\rm for \,\,\,some\,\,\,}
\mathcal{M}, \mathcal{N} \in \tails A
\}.
$$
\begin{definition}
	A noetherian connected graded algebra $A$ is called a {\it noncommutative graded isolated singularity}
	if $\tails A$ has finite global dimension.
\end{definition}

We may define that $\Projn A$ is ``smooth'' if $A$ is a noncommutative graded isolated singularity. 

For an additive category $\sC$, we denote by $\Ind \sC$ the set of isomorphism classes of indecomposable objects.  

\begin{lemma} \label{lem.core} 
	Let $S$ be a quantum polynomial algebra, $f\in S_2$ a regular normal element, and $A=S/(f)$.  Suppose that $X=A/gA$ for some $g\in K_f$. 
	\begin{enumerate}
		\item{} $X\in \Ind \CM^0(A)$. 
		\item{} $\Hom_A(X, A(i))=0$ for every $i\leq 0$. 
	\end{enumerate}
\end{lemma} 

\begin{proof} (1) If $g^2=f$, then 
	$$\begin{CD} \cdots @>g\cdot>> A(-2) @>g\cdot>> A(-1)  @>g\cdot>> A @>g\cdot>> A(1) @>g\cdot>>\cdots \end{CD}$$
	is a complete resolution of $X$, so $X\in \Ind \CM^0(A)$ (cf. \cite[Lemma 5.7, Remark 5.5]{MU3}). 
	
	(2) The exact sequence 
	$$\begin{CD} A(-1) @>g\cdot >> A @>>> X @>>> 0 \end{CD}$$ 
	induces an exact sequence 
	$$\begin{CD} 0 @>>> \Hom_A(X, A(i)) @>>> A_i @>\cdot g>> A_{i+1}, \end{CD}$$ so $\Hom_A(X, A(i))=0$ for every $i\leq 0$. 
\end{proof} 

\begin{definition}
	Let $\mathscr{T}$ be a $k$-linear triangulated category.
	\begin{enumerate}
		\item{} An object $E$ of $\mathscr{T}$ is {\it exceptional} if
		$\End_{\mathscr{T}}(E)=k$ and $\Hom_{\mathscr{T}}(E,E[q])=0$
		for every $q \neq 0$.
		\item{} A sequence of objects $\{ E_{0}, \cdots, E_{l-1} \}$ in $\mathscr{T}$ is an
		{\it exceptional sequence} if
		$E_{i}$ is an exceptional object for every $i=0, \cdots, l-1$, and
		$\Hom_{\mathscr{T}}(E_{j},E_{i}[q])=0$ for every $q$ and every $0 \leq i < j \leq l-1$.
		\item{} An exceptional sequence $\{ E_{0}, \cdots, E_{l-1} \}$ in $\mathscr{T}$ is {\it full}
		if the smallest strictly full triangulated subcategory of $\mathscr{T}$ containing $E_{0}, \cdots, E_{l-1}$
		is equal to $\mathscr{T}$.
		\item{} An exceptional sequence $\{ E_{0}, \cdots, E_{l-1} \}$ in $\mathscr{T}$ is {\it strong}
		if $\Hom_{\mathscr{T}}(U,U[q])=0$ for every $q \neq 0$ where $U=E_{0} \oplus \cdots \oplus E_{l-1}$.
	\end{enumerate}
\end{definition}

Let $k\widetilde {A_1}$ be the path algebra of the quiver
$$
\xymatrix{
1 \ar@<0.5ex>[r] \ar@<-0.5ex>[r]& 2 & \text{($\widetilde {A_1}$ type)},
}
$$
and let $k\widetilde {D_4}$ be the path algebra of the quiver
$$
\xymatrix @R=.2pc{
1 \ar[dr]& & 2\ar[dl] &\\
 & 5 &   & \text{($\widetilde {D_4}$ type).} \\ 
3 \ar[ur]& & 4 \ar[ul] &
}
$$

\begin{theorem} Let $S$ be a 3-dimensional Calabi-Yau quantum polynomial algebra, $0\neq f\in Z(S)_2$, and $A=S/(f)$. If $A$ is a noncommutative graded isolated singularity, then exactly one of the following holds: 
	\begin{enumerate}
		\item{} The following are equivalent: 
		\begin{enumerate}
			\item[(1-a)] $C(A)\cong M_2(k)$. 
          \item[(1-b)] $E_A\subset \PP^2$ is a smooth conic.  
			\item[(1-c)] $E_{A^!}=\emptyset$. 
			\item[(1-d)] $A$ is commutative.  
			\item[(1-e)] $f$ is irreducible. 
			\item[(1-f)] $\sD^b(\tails A)\cong \sD^b(\mod k\widetilde {A_1})$.
		\end{enumerate} 
		\item{} The following are equivalent: 
		\begin{enumerate}
			\item[(2-a)] $C(A)\cong k^4$.
          \item[(2-b)] $\#(E_A)=6$.  
			\item[(2-c)] $\#(E_{A^!})=4$. 
			\item[(2-d)] $A$ is not commutative.
			\item[(2-e)] $f$ is reducible.
			\item[(2-f)] $\sD^b(\tails A)\cong \sD^b(\mod k\widetilde {D_4})$.
		\end{enumerate}
	\end{enumerate}
\end{theorem} 

\begin{proof} Since $A$ is a noncommutative graded isolated singularity, $C(A)$ is a 4-dimensional semi-simple algebra by \cite[
Theorem 5.6]{SV}, so either $C(A)\cong M_2(k)$, or $C(A)\cong k^4$. 
	
	By Theorem \ref{thm.main2}, (1-a) $\Leftrightarrow$ (1-b) $\Leftrightarrow$ (1-c) $\Leftrightarrow$ (1-d) $\Leftrightarrow$ (1-e), and (2-a) $\Leftrightarrow$ (2-b) $\Leftrightarrow$ (2-c) $\Leftrightarrow$ (2-d) $\Leftrightarrow$ (2-e).  
	
	Since $\sD^b(\mod k\widetilde {A_1})\not \cong \sD^b(\mod k\widetilde {D_4})$, it is enough to show that (1-a) $\Rightarrow$ (1-f) and  (2-a) $\Rightarrow$ (2-f).   
	
(1-a) $\Rightarrow$ (1-f): If $C(A)\cong M_2(k)$, then $A\cong k[x, y, z]/(x^2+y^2+z^2)$ by Theorem \ref{thm.main2}, so $\sD^b(\tails A)\cong \sD^b(\mod k\widetilde A_1)$ by \cite[Example 3.23 (1)]{U}.  
	
(2-a) $\Rightarrow$ (2-f): 	If $C(A)\cong k^4$, then $\#(K_f)=4$ so that $K_f=\{g_1, g_2, g_3, g_4\}$ by Theorem \ref{thm.main2}, so $\operatorname{Ind}\uCM^0(A)\supset \{X_1, X_2, X_3, X_4\}$ where $X_i:=A/g_iA$ by Lemma \ref{lem.core}.  Since $\uCM^0(A)\cong \mod C(A)\cong \mod k^4$ by \cite[Lemma 4.13]{MU2}, 
	$\operatorname{Ind}\uCM^0(A)=\{X_1, X_2, X_3, X_4\}$. Since $\uCM^{\ZZ}(A)\cong \sD^b(\mod C(A))$ by \cite[Lemma 4.13]{MU2}, $\{X_1, X_2, X_3, X_4\}$ is a full exceptional sequence for $\uCM^{\ZZ}(A)$.  Since there exists a fully faithful functor $\Phi:\uCM^{\ZZ}(A)\to \sD^b(\tails A)$ such that $\sD^b(\tails A)=\<\cA, \Phi \uCM^{\ZZ}(A)\>$ is a semi-orthogonal decomposition by Theorem \ref{thm.Or} and $\Phi(X_i)=\cX_i$ by Lemma \ref{lem.core} and \cite [Lemma 2.6]{U}, we see that $\{\cA, \cX_1, \cX_2, \cX_3, \cX_4\}$ is a full exceptional sequence for $\sD^b(\tails A)$.  
	
	Since $A$ is a noetherian AS-Gorenstein algebra of dimension 2 and of Gorenstein parameter 1 by Lemma \ref{lem.asg} and $\gldim (\tails A)=2$ (cf. \cite[Lemma 2.2]{U}), 
	\begin{align*}
	& \Hom_{\cA}(\cA, \cA)=k \\
	& \Hom_{\cA}(\cX_i, \cX_j)=\begin{cases} k & \textnormal { if } i=j \\ 0 & \textnormal { if } i\neq j \end{cases} \\
	& \Hom_{\cA}(\cA, \cX_i)=\Hom_{A}(A, X) =X_0=k \\
	& \Hom_{\cA}(\cX_i, \cA)=\Hom_A(X, A)=0  
	\end{align*}
	and 
	\begin{align*}
	& \Ext^1_{\cA}(\cX_i, \cX_j)=0 \\
	& \Ext^1_{\cA}(\cA, \cX)=D\Hom_{A}(X, A(-1)) =0  \\
	& \Ext^1_{\cA}(\cX, \cA)=D\Hom_{A}(A, X(-1))= D(X_{-1})=0 \\
	& \Ext^q_{\cA}(\cX, \cY)=0 \; \; \;  \textnormal { if } \;   q\geq 2.  
	\end{align*} 
	for $\cX, \cY\in \{\cA, \cX_1, \cX_2, \cX_3, \cX_4\}$ by \cite[Lemma 2.3]{U} and Lemma \ref{lem.core}.   It follows that $\{\cA, \cX_1, \cX_2, \cX_3, \cX_4\}$ is a full strong exceptional sequence for $\sD^b(\tails A)$, and 
	$$\End_{\cA}(\cA\oplus \cX_1\oplus \cX_2\oplus \cX_3\oplus \cX_4)\cong k\widetilde {D_4},$$ so $\sD^b(\tails A)\cong \sD^b(\mod k\widetilde {D_4})$.   
\end{proof}


\end{document}